\documentclass[12pt]{article}

\usepackage[T1]{fontenc}
\usepackage[utf8]{inputenc}
\usepackage{microtype}
\usepackage{geometry}
\usepackage{titlesec}
\titleformat{\subsection}[runin]{\bf}{\thesubsection.}{3pt}{}
\usepackage{amsmath, amsfonts, amssymb, tikz-cd, mathtools}
\usepackage{graphicx}
\usepackage{caption, subcaption}
\usepackage{enumitem}

\usepackage{color}
\usepackage[english]{babel}
\usepackage{float}
\usepackage[toc,page]{appendix}
\usepackage{listings}
\usepackage{lscape}
\usepackage{pgfplots}
\pgfplotsset{compat=1.15}
\usepackage{mathrsfs}
\usetikzlibrary{arrows}
\usepackage{import}
\usepackage{xifthen}
\usepackage{pdfpages}
\usepackage{transparent}
\usepackage{hyperref}
\usepackage{mathtools}

\usepackage{amscd}

\usepackage{amsthm}
\newtheorem{theorem}{Theorem}[section]
\newtheorem{lemma}[theorem]{Lemma}
\newtheorem{proposition}[theorem]{Proposition}
\newtheorem{corollary}[theorem]{Corollary}

\theoremstyle{definition}
\newtheorem{definition}[theorem]{Definition}

\newtheorem*{question*}{Question}

\theoremstyle{remark}
\newtheorem{remark}[theorem]{Remark}
\newtheorem{example}[theorem]{Example}

\usepackage{bbm}
\usepackage{hyperref}
\definecolor{myPurp}{RGB}{105,73,148}
\definecolor{myRed}{RGB}{228,37,53}
\definecolor{myCyan}{RGB}{0,95,114}
\hypersetup{
     colorlinks=true,
     linkcolor = myRed,
     citecolor = myPurp,      
     urlcolor= myCyan,
     }

\usepackage[
backend=biber,
style=iso-alphabetic,
sorting=nyt
]{biblatex}
\usepackage{csquotes}

\numberwithin{equation}{section}

\newcommand{\Ham}{\mathrm{Ham}}
\newcommand{\stab}{\mathrm{Stab}}
\newcommand{\Diff}{\mathrm{Diff}}

\newcommand{\Homeo}{\mathrm{Homeo}}

\newcommand{\Frederic}{Fr\'{e}d\'{e}ric Le Roux}
\newcommand{\Humiliere}{Humili\`ere}

\title{\bf Balanced curves, quasimorphisms, and the equator conjecture}
\author{\sc Yongsheng Jia \and \sc Richard Webb}
\date{}

\begin{document}

\maketitle

\begin{abstract}

We construct a new infinite-dimensional family of homogeneous quasimorphisms on the group of Hamiltonian diffeomorphisms of the two-sphere. 
Moreover, for any constant $K$ less than the total area of the sphere, we produce unbounded homogeneous quasimorphisms that vanish on any map supported on some disk of area at most $K$.
As an application, we prove an analogue of the equator conjecture, namely that the space of equators equipped with any choice of quantitative fragmentation metric has infinite diameter.
To prove our results, we introduce tools that draw inspiration from the theory of curve graphs used to study mapping class groups.
\end{abstract}

\noindent{\it Keywords}: Quasimorphisms, Hamiltonian diffeomorphisms, group actions

\tableofcontents

\section{Introduction}\label{section-Introduction}
Let $\omega$ be a symplectic form on $S^2$ of total area equal to $1$. In this paper, we study the smooth simple closed curves that separate $S^2$ into two components of equal area (called the \emph{equators}), and their interaction with the group of Hamiltonian diffeomorphisms $\Ham(S^2)$. Our methods also apply in the non-smooth setting, i.e. the related group of (isotopically-trivial) homeomorphisms of $S^2$ that preserve the measure induced by $\omega$, denoted by $\Homeo_0(S^2,\omega)$.
In fact, the former group is $C^0$-dense in the latter, see \cite{oh2006c,sikorav2007approximation,munkres1960obstructions}.

The group $\Ham(S^2)$ acts naturally on the space of equators, and it is well known that the action is transitive. 
A norm on $\Ham(S^2)$ (with the appropriate invariance properties) enables us to define a (pseudo-)metric on the space of equators, by defining the distance between $\alpha$ and $\beta$ to be the infimum of the set of norms of $f\in\Ham(S^2)$, such that $f(\alpha)=\beta$.
This can be thought of as the energy required to move $\alpha$ to $\beta$.
But there are different choices of norm (or metric) on $\Ham(S^2)$.
We will discuss the quantitative fragmentation metric and the Hofer metric.

Fix a constant $A$ such that $0<A<1$. For $f\in \Ham(S^2)$, the $A$\emph{-quantitative fragmentation norm} $|f|_A$ is defined to be the word length of $f$ with respect to the (infinite and conjugation-invariant) generating set given by all area-preserving maps supported on disks of area at most $A$. 
The $A$-\emph{quantitative fragmentation metric} on $\Ham(S^2)$ is defined by setting $d_A(f,g)=|f^{-1}g|_A$.
This is a bi-invariant metric, and can be similarly defined for $\Homeo_0(S^2,\omega)$.
This metric is intimately related to the algebraic structure of these groups,  see for example \cite{banyaga1978structure,LeRoux}.
In \cite{burago2008conjugation}, the quantitative fragmentation norm is discussed for the torus (in the displaceable case), and is shown to have infinite diameter.
In fact, this is also true when the fundamental group is complicated enough, as shown by Brandenbursky and K\k edra \cite{brandenbursky2022}.
The case of the sphere is therefore somewhat curious. We prove the following:

\begin{theorem}\label{alternative equator conjecture}
For every $0<A<1$, the space of equators on $S^2$ equipped with the $A$-quantitative fragmentation metric has infinite diameter.
\end{theorem}

The theorem also holds in the non-smooth setting, namely, that the $\Homeo_0(S^2,\omega)$-orbit of the standard equator of $S^2$, has infinite diameter.
The infinite diameter statement for the space of equators is at least as hard to establish as it is for the groups $\Ham(S^2)$ or $\Homeo_0(S^2,\omega)$.
Indeed, as a consequence of the above, the $A$-quantitative fragmentation metric on these groups has infinite diameter for any choice of $A$.
For the displaceable case $0<A<\frac{1}{2}$, that $\Ham(S^2)$ and $\Homeo_0(S^2,\omega)$ have infinite diameter is not new, and was first proved by Cristofaro-Gardiner, \Humiliere, Mak, Seyfaddini, and Smith in \cite{cristofaro2022quantitative}.

Recently Brandenbursky and Shelukhin proved that the space of equators with the $L^p$ metric has infinite diameter, for all $p\geq 1$, see \cite{brandenbursky2025lpdiameterspacecontractibleloops}. Both this result and Theorem~\ref{alternative equator conjecture} are analogues of the equator conjecture, discussed below.

We now discuss the Hofer metric \cite{hofer1990topological}.
When viewed as a norm on $\Ham(S^2)$, this is a measure of the required total variance of a generating Hamiltonian, in the sense of an integral of the difference between its maxima and minima.
In particular, the absence of derivatives in its definition makes the Hofer metric/norm rather subtle, and so, the fact that it induces a non-degenerate metric is remarkable. Hofer proved this for the Hamiltonian diffeomorphism group of $\mathbb{R}^{2n}$.
For $S^2$ (and tame, rational symplectic manifolds) it was proved by Polterovich \cite{polterovich1993symplectic}, and for general symplectic manifolds by Lalonde--McDuff \cite{lalonde1995geometry}. For more background on the Hofer metric, we refer the reader to \cite{mcduff2017introduction} and the references therein.

With the Hofer metric, $\Ham(S^2)$ has infinite diameter, a result originally due to Polterovich \cite{polterovich1998hofer}.
In fact, the Hofer geometry of $\Ham(S^2)$ is now known to contain quasi-flats of arbitrary dimension, see \cite{cristofarogardiner2024,polterovich2023lagrangian}.
On the other hand, the Hofer metric on the space of equators is more mysterious; indeed, the equator conjecture asserts that its Hofer diameter is infinite \cite[Problem~32]{mcduff2017introduction} and this is still open.

For the unit disk, Khanevsky proved that the space of \emph{diameters} has infinite Hofer diameter \cite{khanevsky2009hofer}. The strategy of Khanevsky's proof is to find an unbounded Hofer-Lipschitz quasimorphism that vanishes on stabilisers of diameters. This turns out to be a useful framework for showing spaces of Lagrangians have infinite Hofer diameter.

In this paper we also construct and make use of quasimorphisms. 
The general problem of constructing quasimorphisms has played an important role in geometric group theory and symplectic geometry since Gromov introduced the notion of bounded cohomology in the early 80s \cite{gromov1982volume}. 
See also \cite{bavard1991longueur,calegari2009scl,ghys2007,kotschick2004} and the references therein.

There is an interesting history of quasimorphisms defined on $\Ham(S^2)$.
Entov--Polterovich \cite{entov2003calabi} extended the Calabi homomorphism (on any displaceable open disk) to a \emph{Calabi} quasimorphism on $\Ham(S^2)$.
This is a rather beautiful result, especially in light of the fact that the Calabi homomorphism cannot be extended to a homomorphism on $\Ham(S^2)$ (because $\Ham(S^2)$ is simple \cite{banyaga1978structure}).
Gambaudo--Ghys \cite{gambaudo2004commutators} constructed an infinite-dimensional vector space of homogeneous quasimorphisms on $\Ham(S^2)$.
Cristofaro-Gardiner, \Humiliere, Mak, Seyfaddini, and Smith \cite{cristofaro2022quantitative} proved that there is an infinite-dimensional vector space of $C^0$-continuous and Hofer-Lipschitz quasimorphisms on $\Ham(S^2)$, answering questions of Entov--Polterovich--Py \cite{entov2012continuity}, and these extend to $\Homeo_0(S^2,\omega)$.

We construct a new family of quasimorphisms on $\Ham(S^2)$ and $\Homeo_0(S^2,\omega)$ to prove Theorem~\ref{alternative equator conjecture}.
In fact, as we shall see, it turns out to be useful to consider the action on sets of curves that are larger than the set of equators.
Let $\varepsilon>0$. 
We say that a simple closed curve of $S^2$ is $\varepsilon$\emph{-balanced} if both complementary components have area at least $\varepsilon$. For example, an equator is $\varepsilon$-balanced, for any $0<\varepsilon\leq \frac{1}{2}$.

\begin{theorem}\label{Q(Diff) is infinite dimensional}
For any $\varepsilon>0$, the vector space of $C^0$-continuous homogeneous quasimorphisms on $\Ham(S^2)$, and $\Homeo_0(S^2,\omega)$, that vanish on stabilisers of $\varepsilon$-balanced curves, is infinite dimensional.
\end{theorem}

In particular, our quasimorphisms can be chosen to vanish on stabilisers of equators in $S^2$.
The existence of an unbounded Hofer-Lipschitz quasimorphism that vanishes on stabilisers of equators would prove the equator conjecture, see \cite{khanevsky2009hofer,serraille2025link} or Section~\ref{section-Alternative equator theorem}. 
Unfortunately, there seems to be tension between vanishing on stabilisers of equators and being Hofer-Lipschitz.
At time of writing, we do not know of a non-trivial example of a quasimorphism that we construct, which is proved to be (or not to be) Hofer-Lipschitz. 
The argument of Khanevsky~\cite{khanevsky2024quasimorphisms} does not readily apply to show that our quasimorphisms are not Hofer-Lipschitz, and there is a possibility that at least some of them are.
On the other hand, Serraille and Trifa \cite{serraille2025link} showed that some linear combinations of the Hofer-Lipschitz quasimorphisms in \cite{cristofaro2022quantitative} vanish on the stabilisers of equators.
Unfortunately, at present, it is not known whether these quasimorphisms vanish completely on $\Ham(S^2)$ or not.

\subsection*{Methodology and outline of proofs.} 
The quasimorphisms in Theorem~\ref{Q(Diff) is infinite dimensional} are constructed via group actions by isometries on Gromov hyperbolic graphs. Specifically, we make use of a criterion developed by Bestvina--Fujiwara~\cite{bestvina2002bounded} (see also~\cite{epstein1997second, fujiwara1998second}) which in fact constructs infinitely many quasimorphisms. 
Bestvina and Fujiwara showed that their criterion applies to the mapping class group by analysing its action on the curve graph,
see \cite{bestvina2002bounded} for more details. 
We discuss their criterion in Section~\ref{section-BF quasimorphisms}.

To apply the Bestvina--Fujiwara criterion to our setting, we must construct a Gromov hyperbolic and infinite-diameter graph for $\Ham(S^2)$ to act on by isometries (in fact, we provide uncountably many such graphs).
The graphs we introduce are analogues of the curve graphs for mapping class groups. 
Curve graphs were originally introduced by Harvey \cite{harvey1981boundary}.
Key properties of the curve graphs were established in the seminal work of Masur and Minsky \cite{Masur1999, masur2000geometry}, and now play a central role in the study of mapping class groups. 
In particular, it is shown in \cite{Masur1999} that the curve graph is Gromov hyperbolic, and that pseudo-Anosovs act as hyperbolic isometries. 

We also take inspiration from the non-conservative diffeomorphism groups.
Bowden, Hensel, and the second author \cite{bowden2022quasi}, introduced the \emph{fine curve graph}, and used the Bestvina--Fujiwara criterion to construct quasimorphisms on $\Diff_{0}(S_{g})$ for each $g\geq 1$.
This method breaks down for $S^2$ in the non-conservative setting (it must since $\Diff_0(S^2)$ is uniformly perfect \cite{burago2008conjugation,tsuboi2008uniform}).
This is because the fine curve graph of $S^2$ is either empty (there are no essential curves), or even if you permit inessential curves as vertices, this ends up having finite diameter (and so the Bestvina--Fujiwara criterion cannot apply).
However, if one equips $S^2$ with a symplectic form $\omega$, and fixes $0<\varepsilon\leq \frac{1}{2}$, and considers the induced subgraph of vertices corresponding to $\varepsilon$-balanced curves of $S^2$, then we obtain graphs with infinite diameter.
In this paper, we also put in edges between two vertices (that is, $\varepsilon$-balanced curves), when they intersect transversely twice
(This in particular ensures that the graph is connected for $\varepsilon=\frac{1}{2}$).
This defines the $\varepsilon$\emph{-balanced curve graph} $\mathcal{C}_\varepsilon^\dagger(S^2)$, on which $\Ham(S^2)$ acts by isometries. See Section~\ref{section-New graph} for the definition.

The connectivity of the graphs $\mathcal{C}_\varepsilon^\dagger(S^2)$ is proved in Section~\ref{Subsection-connectivity}. That they have infinite diameter is a consequence of Lemma~\ref{hyperbolics}. That these graphs are Gromov hyperbolic is proved in Section~\ref{Subsection-hyperbolicity}. Section~\ref{section-Hyperbolicity of the balanced curve graph} is perhaps the most technical part of the paper.

\subsection*{Constructing quasimorphisms}
Let $\varepsilon\in (0,\frac{1}{2}]$. To apply the Bestvina--Fujiwara criterion to $\Ham(S^2)$ acting on $\mathcal{C}_\varepsilon^\dagger(S^2)$, we require two hyperbolic isometries that satisfy a strong notion of independence. We provide a flexible construction of Hamiltonian diffeomorphisms that act hyperbolically on $\mathcal{C}_\varepsilon^\dagger(S^2)$. 
Roughly speaking, these are area-preserving diffeomorphisms that represent pseudo-Anosov mapping classes on (connected) subsurfaces of $S^2$, but we require the subsurface and each of its complementary disks in $S^2$ to have sufficiently small area in terms of $\varepsilon$. This is the content of Lemma~\ref{hyperbolics}.

In Section~\ref{section-BF quasimorphisms} we prove Theorem~\ref{Q(Diff) is infinite dimensional}.
There, we address the problem of constructing independent hyperbolic isometries of $\mathcal{C}_\varepsilon^\dagger(S^2)$.
To do this, we exploit the fact that when $0<\varepsilon'<\varepsilon$, there is a $1$-Lipschitz inclusion $\mathcal{C}_\varepsilon^\dagger(S^2)\rightarrow \mathcal{C}_{\varepsilon'}^\dagger(S^2)$.
Our construction of hyperbolic elements allows us to find elements that act on $\mathcal{C}_{\varepsilon'}^\dagger(S^2)$ hyperbolically (and therefore also on  $\mathcal{C}_{\varepsilon}^\dagger(S^2)$), and elements that act hyperbolically on 
$\mathcal{C}_{\varepsilon}^\dagger(S^2)$ but elliptically on $\mathcal{C}_{\varepsilon'}^\dagger(S^2)$.
Such a pair of elements must be independent, and so we can apply the Bestvina--Fujiwara criterion. 
That the quasimorphisms vanish on stabilisers of $\varepsilon$-balanced curves follows from the construction of the (homogenisation of the) Bestvina--Fujiwara quasimorphisms.
That they are $C^0$-continuous is then a consequence of~\cite{entov2012continuity}, where it is also shown that such quasimorphisms extend to $\Homeo(S^2,\omega)$.

Our proof that Theorem~\ref{alternative equator conjecture} follows from Theorem~\ref{Q(Diff) is infinite dimensional} is given in Section~\ref{section-Alternative equator theorem}. There, we also explain why the equator conjecture follows if one of our (non-trivial) quasimorphisms is Hofer-Lipschitz.
\section*{Acknowledgements}
The first author wants to thank Mladen Bestvina for introducing Sobhan Seyfaddini's work to him during the conference Group Actions and Low-Dimensional Topology in El Barco de Ávila.
The authors thank Jonathan Bowden, Michael Brandenbursky, Sebastian Hensel, Vincent \Humiliere,  Michael Khanevsky, \Frederic, Yusen Long, Pierre Py, Sobhan Seyfaddini, and Egor Shelukhin for helpful discussions. The first author is supported by EPSRC DTP EP/V520299/1.

\section{Preliminaries}\label{section-Preliminaries}

\subsection{Graphs as metric spaces}
Let $\Gamma$ be a connected graph. We will view $\Gamma$ as a metric space where the underlying set is the vertices of $\Gamma$, and the distance between $x$ and $y$ is defined by the length of a shortest path in $\Gamma$ connecting $x$ and $y$. In this paper we will automatically think of (connected) graphs as metric spaces without having to clarify this.

A \emph{geodesic} $P$ is a sequence of vertices $\{v_i\}_i$ such that $d_{\Gamma}(v_i,v_j)=|i-j|$ for all $i,j$.
A \emph{quasi-geodesic} is a sequence $\{v_i\}_i$ such that there exist $K$ and $L$ with 
$$
\frac{1}{K}|i-j|-L\leq d_{\Gamma}(v_i,v_j)\leq K|i-j|+L.
$$
More specifically, we say that $\{v_i\}_i$ is a $(K,L)$-quasi-geodesic, or less specifically, a $C$-quasi-geodesic when it is a $(C,C)$-quasi-geodesic.

In general, let $(X,d)$ be a metric space, and let $A$ be a subset of $X$. We denote by $N_{K}(A)$ the $K$-neighborhood of $A$ in $X$, i.e. 
$$
N_{K}(A)\coloneqq \{y\in X | \text{ } d(a,y)\leq K \text{ for some }a\in A\}.
$$

\subsection{Gromov hyperbolicity}
As we saw above, connected graphs determine metric spaces where each pair of vertices is connected by a geodesic. We now define what it means for a graph to be Gromov hyperbolic.
\begin{definition}\label{slim triangle}(Slim triangles)
We say that a  graph $(\Gamma,d_{\Gamma})$ is \emph{hyperbolic} if there exists $\delta\geq0$ such that for any three geodesics $P_1,P_2,P_3$ that form a triangle in $\Gamma$, we have that $P_1\subset N_{\delta}(P_2\cup P_3)$. More specifically, we say that $\Gamma$ is \emph{$\delta$-hyperbolic} or $\Gamma$ \emph{has $\delta$-slim triangles}.
\end{definition}
The definition above applies equally well to any geodesic metric space, but we will not require this generality in this paper. We refer the reader to \cite{alonso1991notes} for alternative definitions of hyperbolicity as well as proofs of their equivalence. 

We make use of the following remarkable and useful criterion for proving the hyperbolicity of a graph, known as the `Guessing Geodesics Lemma'. This was originally discovered by Masur and Schleimer, with alternative proofs by Bowditch, see \cite{masur2013geometry,bowditch2014uniform}.
\begin{lemma}\label{Guessing Geodesic Lemma}
Let $(\Gamma,d_{\Gamma})$ be a graph. If there exists some constant $\lambda\geq 0$, such that for all $x,y\in \Gamma$, we can find a connected subgraph $\mathcal{L}(x,y)\subset \Gamma$ containing $x,y$ such that
\begin{itemize}
    \item $\forall x,y,z \in \Gamma$, we have $\mathcal{L}(x, y) \subset N_{\lambda}(\mathcal{L}(x, z) \cup \mathcal{L}(z, y))$, and
    \item $\forall x,y \in \Gamma$ with $d_{\Gamma}(x,y)\leq 1$, the diameter of $\mathcal{L}(x,y)$ in $\Gamma$ is at most $\lambda$. 
\end{itemize}
Then $\Gamma$ is hyperbolic for some $\delta \geq 0$ depending only on $\lambda$.

\end{lemma}

\begin{remark}
The connected subgraph $\mathcal{L}(x,y)$ associated to $x,y\in \Gamma$ is called the `guessing geodesic' between $x$ and $y$, and is known to have uniformly bounded Hausdorff distance from a geodesic between $x$ and $y$ (bounded in terms of $\lambda$).
In fact, $\delta$ can be chosen as any number greater than or equal to $(3m-10\lambda)/2$, where $m$ is any positive real number satisfying 
$$
2\lambda\left(6+\log _{2}(m+2)\right) \leq m.
$$
\end{remark}
Hensel--Przytycki--Webb constructed a path via the `unicorn arcs' to show the uniform hyperbolicity of the arc graphs and curve graphs of compact oriented surfaces \cite{hensel20151}. 
Later, Przytycki--Sisto constructed a path via `bicorn curves' to show the uniform hyperbolicity of the curve graphs of closed oriented surfaces of genus at least $2$ \cite{przytycki2017note}.
The arguments used in both papers to prove uniform hyperbolicity are to show that the paths given by `unicorn arcs' or `bicorn curves' satisfy the conditions in Lemma \ref{Guessing Geodesic Lemma}.
In this paper, we will also make use of Lemma~\ref{Guessing Geodesic Lemma} for the graphs we introduce.

Now suppose a group $G$ acts on a hyperbolic space $(X,d)$ by isometries. For $g\in G$ and $x\in X$ we define 
$$
|g|:= \lim_{k\rightarrow \infty}\frac{1}{k}d(x,g^{k}x),
$$
to be the \emph{asymptotic translation length} of $g$. In fact, this limit exists and is independent of the choice of $x$. We actually have the following classifications of isometries \cite{gromov1987hyperbolic}.

\begin{definition}
Let $g$ be an isometry of $X$. We say that $g$ is 
\begin{itemize}
    \item an elliptic element if the orbit of $\{g^i\cdot x\}_{i\in \mathbb{Z}}$ is bounded for some (and hence, for every) $x\in X$;
    \item a parabolic element if $|g|=0$ and every orbit of $g$ is unbounded;
    \item a hyperbolic (or loxodromic) element if $|g|>0$.
\end{itemize}
\end{definition}
If $g$ is a hyperbolic (or loxodromic) element, then any orbit of $g$ is a $C$-quasi-axis, i.e. a $g$-invariant $C$-quasi-geodesic, for some $C$ depending on the orbit.
And $g$ acts on this quasi-axis by translation.
We typically write $A_g$ for an invariant quasi-axis of $g$.
The positive asymptotic translation length is indeed a characteristic of a hyperbolic (or loxodromic) element in the context of a group acting on a metric space.
This characteristic helps distinguish hyperbolic elements from elliptic and parabolic elements, each of which behaves differently under the group action. 
For details, we refer the reader to have a look at \cite{bridson2013metric} or \cite{gromov1987hyperbolic}.

\subsection{Quasi-isometry}
\begin{definition}\label{qi}
 Let $\left(X, d_{X}\right)$ and $\left(Y, d_{Y}\right)$ be metric spaces. A map $f\colon X \rightarrow Y$ is a \emph{$(K, C)$-quasi-isometric embedding} if there exists constants $K \geq 1$ and $C \geq 0$ so that for all $x, x^{\prime} \in X$ we have
$$
\frac{1}{K} d_{X}(x, x^{\prime})-C \leq d_{Y}(f(x), f(x^{\prime})) \leq K d_{X}(x, x^{\prime})+C .
$$
Furthermore, $f$ is called a \emph{quasi-isometry} if it is also \emph{coarsely surjective}, i.e., there exists $r \geq 0$ so that for all $y \in Y$ there exists $x \in X$ such that $d_{Y}(f(x),y)\leq r$. In this case, we say $X$ and $Y$ are quasi-isometric.
\end{definition}
We can give an alternative definition of a hyperbolic isometry. Let $G$ be a group that acts on a hyperbolic space $(X,d)$ by isometries. Then $g\in G$ is a hyperbolic isometry of $X$ if and only if $k\mapsto g^kx$ is a quasi-isometric embedding of $\mathbb{Z}$ into $X$, for every $x\in X$.

\subsection{Arc and curve graphs}
\label{subsection-arc and curve graphs}

Let $S=S_{g,p}$ be the connected, compact, oriented surface of genus $g$ with $p$ boundary components. We write $\partial S$ for the boundary of $S$. We define $\xi(S):= 3g+p-3$ and $\xi(S)$ is called the complexity of $S$.

\begin{definition}
Let $S=S_{g,p}$ be a surface with boundary and let $\xi(S)\geq 1$. A simple closed curve in $S$ is \emph{essential} if it is not null-homotopic (equivalently, does not bound a disk in $S$) and \emph{non-peripheral} if it is not homotopic into $\partial S$. A \emph{curve} of $S$ is a homotopy class of essential, non-peripheral, simple closed curve of $S$. An \emph{arc} of $S$ is an ambient isotopy class of a properly embedded interval in $S$ which does not cut off a disk of $S$ (equivalently, it is the homotopy class through such properly embedded intervals). The \emph{arc and curve graph} for $S=S_{g,p}$, denoted by $\mathcal{AC}(S)$, is the graph whose vertex set $V(\mathcal{AC}(S))$ is the set of arcs and curves of $S$. Two vertices are joined by an edge if and only if the corresponding classes admit disjoint representatives. The \emph{arc graph} $\mathcal{A}(S)$ is the induced subgraph of $\mathcal{AC}(S)$  on the vertices that are arcs of $S$, and the \emph{curve graph} $\mathcal{C}(S)$ is the induced subgraph of $\mathcal{AC}(S)$  on the vertices that are curves.
\end{definition}

\begin{remark}
For $S_{0,4}$, we can also define its curve graph after a slight modification. Vertices are still the isotopy classes of essential, non-peripheral curves in $S_{0,4}$. However, two vertices are joined by an edge if and only if the corresponding curves have geometric intersection number equal to $2$. The graph $\mathcal{C}(S_{0,4})$ is isomorphic to the Farey graph. It is known that the Farey graph is $1$-hyperbolic, see for example~\cite{minsky1996geometric,leasure2002geodesics}.
\end{remark}
For a surface $S$ with $\xi(S)\geq 2$, the curve graph $\mathcal{C}(S)$ is proved to be connected, of infinite diameter and hyperbolic by Masur and Minsky~\cite{Masur1999}. Later, the uniform hyperbolicity of the curve graph $\mathcal{C}(S)$ was proved independently by several groups of people,~\cite{aougab2013uniform,bowditch2014uniform,clay2015uniform,hensel20151,przytycki2017note}. Here, uniform hyperbolicity means that the hyperbolicity constant $\delta$ can be chosen to be independent of the topological structure of the surface, i.e. independent of genus and number of boundary components. 
The uniform hyperbolicity of 
the arc graphs $\mathcal{A}(S)$ was proved by Hensel--Przytycki--Webb~\cite{hensel20151}.

The mapping class group of $S$, denoted by $\mathrm{MCG}(S)$, is defined to be the group of orientation-preserving homeomorphisms of $S$ modulo isotopy. Observe that $\mathrm{MCG}(S)$ has a natural isometric action on the curve graph $\mathcal{C}(S)$. The following was proved by Masur--Minsky~\cite{Masur1999}.

\begin{proposition}\label{pseudoAnosove mapping classes are loxodromic elements}
The mapping class group $\mathrm{MCG}(S)$ acts on $\mathcal{C}(S)$ by isometries, and $g\in \mathrm{MCG}(S)$ is a hyperbolic (or loxodromic) element if and only if $g$ is pseudo-Anosov.
\end{proposition}
\begin{remark}
For the definition and properties of pseudo-Anosov mapping classes, we refer the reader to \cite{farb2011primer}.
\end{remark}

\subsection{Intersections and dual trees}

Suppose that $\alpha$ and $\beta$ are simple closed curves on a surface. We will write $\alpha \pitchfork \beta$ to indicate that $\alpha$ and $\beta$ are disjoint, or intersect transversely.

\begin{definition}\label{dual tree}
Assume $\alpha$ and $\beta$ are simple closed curves on $S^2$, $\alpha\cap\beta\neq\emptyset$, and $\alpha \pitchfork \beta$. Let $\alpha'\subset\alpha$ and $\beta'\subseteq \beta$ be connected subarcs of $\alpha$ and $\beta$ respectively (and also allow the case $\beta'=\beta$). We define the dual tree for the pair $(\alpha',\beta')$, denoted by $T=T_{\alpha',\beta'}$, as follows: the vertices are all connected components of the complements of $\alpha'\cup \beta'$. And two distinct vertices are joined by an edge if the corresponding components share a connected component of $\beta'\backslash \alpha'$ in common along their boundaries. It is well known that this defines a tree because $S^2\setminus \alpha'$ is an open disk.

\end{definition}
\begin{figure}[H]
    \centering
    \includegraphics[width=0.8\textwidth]{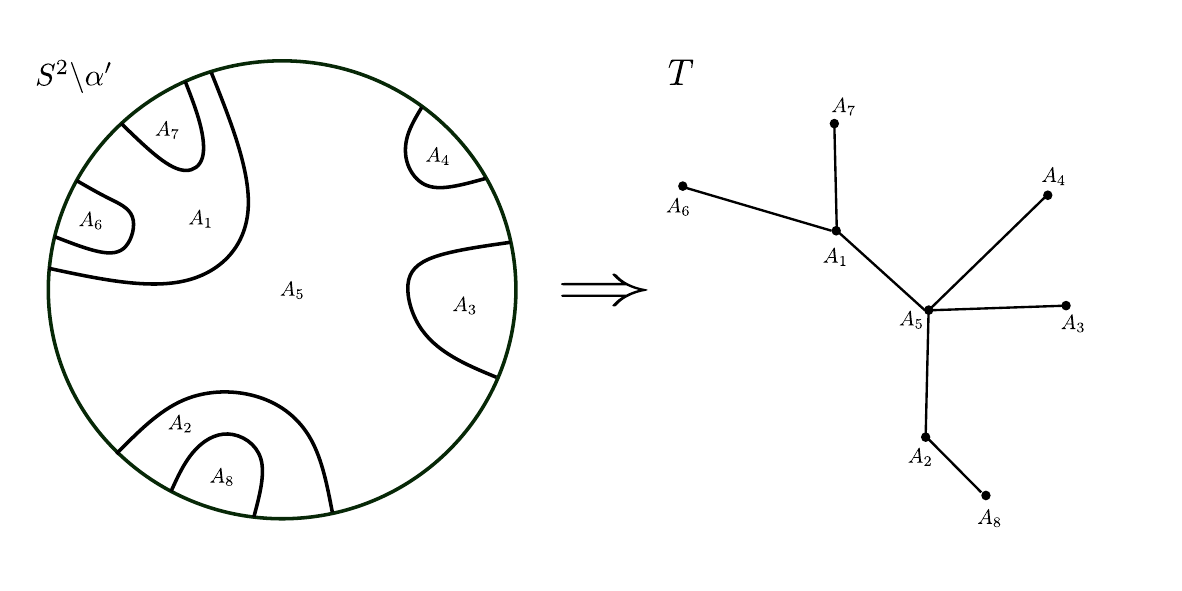}
    \caption{The dual tree}
    \label{fig: Dual Tree}
\end{figure}

\begin{example}
Figure \ref{fig: Dual Tree} provides an example of how we construct the dual tree for given $\alpha'$ and $\beta'$. The left part of Figure \ref{fig: Dual Tree} is a surface which is homeomorphic to $S^2\backslash \alpha'$. The arcs inside the disk are the components of $\beta'\backslash \alpha'$.
\end{example}

\begin{remark}
    When $\alpha'=\alpha$, we can still define a dual tree $T=T_{\alpha',\beta'}$ but there is one dual tree for each (of the two) connected components of $S^2\setminus \alpha'$ (again homeomorphic to an open disk). The definition of the dual tree is the same as above once the component is understood.
\end{remark}
\section{The balanced curve graph}\label{section-New graph}

\subsection{Definition and theorem}
Let $(S^{2},\omega)$ be the two-sphere equipped with an area form $\omega$ of total area equal to one. Fix a constant $\varepsilon$ such that $0<\varepsilon\leq \frac{1}{2}$.

\begin{definition}\label{The definition of the balanced curve graph}
Let $\alpha \subset S^{2}$ be a smooth simple closed curve. We say that $\alpha$ is \emph{$\varepsilon$-balanced} if the areas of both connected components of $S^2\setminus \alpha$ are greater than or equal to $\varepsilon$. The \emph{$\varepsilon$-balanced curve graph}, denoted by $\mathcal{C}^{\dagger}_{\varepsilon}(S^{2})$, is the graph whose vertex set is the set of all $\varepsilon$-balanced curves in $S^{2}$, where two vertices are joined by an edge if and only if their corresponding $\varepsilon$-balanced curves are either disjoint or intersect twice transversely.
\end{definition}

We will prove the following.

\begin{theorem}\label{infinite diameter}
For any $0<\varepsilon\leq \frac{1}{2}$, the graph $\mathcal{C}^{\dagger}_{\varepsilon}(S^{2})$ is connected, has infinite diameter, and is hyperbolic.
\end{theorem}


For now, we will take for granted that the above graphs are connected, and defer the proof of connectivity until Section~\ref{Subsection-connectivity}. We will prove the infinite diameter statement in Lemma~\ref{hyperbolics} in Section~\ref{subsection-witnesses and subsurface projections}, and prove that the graphs are hyperbolic in Section~\ref{Subsection-hyperbolicity}.

\begin{remark}
An alternative definition of $\mathcal{C}^{\dagger}_{\varepsilon}(S^{2})$ is given by the same set of vertices but with edges only between pairs of disjoint curves. When $\varepsilon=\frac{1}{2}$, the alternative graph is not connected. When $0 < \varepsilon < \frac{1}{2}$, the $\varepsilon$-balanced curve graph and the alternative definition are quasi-isometric. In Section~\ref{Subsection-hyperbolicity} we show that there is a choice of hyperbolicity constant that works for all of the $\varepsilon$-balanced curve graphs simultaneously. We do not believe this is the case for the alternative definition for values of $\varepsilon$ tending to $\frac{1}{2}$.\end{remark}

\subsection{Witnesses and subsurface projections}

\label{subsection-witnesses and subsurface projections}

To prove the infinite diameter statement in Theorem~\ref{infinite diameter}, we use the concept of \emph{witnesses}, which were introduced as \emph{holes} in Masur and Schleimer's paper \cite{masur2013geometry}. The idea is that any vertex of our graph must cut a fixed subsurface $W$, \emph{the witness}, in an essential way, which enables one to define a coarsely Lipschitz projection to the (arc and) curve graph of $W$. Our setup is different, so we define and characterize our witnesses here.

\begin{definition}\label{witness}
A (proper) \emph{subsurface} $W\subset S^2$ is a connected $2$-manifold with boundary, embedded in $S^2$. We say a simple closed curve $\alpha\subset S^2$ \emph{cuts} a subsurface $W\subset S^2$ if $\alpha\cap W$ is an essential, non-peripheral curve in $W$, or, contains an essential properly embedded interval in $W$, or both. We say that a subsurface $W\subset S^{2}$ is a \emph{witness} for $\mathcal{C}^{\dagger}_{\varepsilon}(S^{2})$ if every vertex of $\mathcal{C}^{\dagger}_{\varepsilon}(S^{2})$ cuts $W$.
\end{definition}

Let us now consider subsurfaces $W\subset S^2$ by removing $n\geq 4$ open disks from $S^2$ whose closures are pairwise disjoint. These disks are denoted by $A_{1}$, $A_{2}$, \dots, $A_{n}$ respectively. Now we're going to give a sufficient condition for $W$ to be a witness for $\mathcal{C}^{\dagger}_{\varepsilon}(S^{2})$.

\begin{lemma}\label{Sufficient condition for W to be a witness}
If for all $i$ with $1\leq i \leq n$, we have $\mathrm{area}(A_{i})+\mathrm{area}(W)<\varepsilon$, then $W$ is a witness for $\mathcal{C}^{\dagger}_{\varepsilon}(S^{2})$. 

\end{lemma}
\begin{proof}
Let $\alpha$ be an arbitrary $\varepsilon$-balanced curve. Let's assume for a contradiction that $\alpha$ does not cut $W$. This means that either $\alpha$ is disjoint from $W$, or, it intersects $W$ but only with inessential or peripheral components.

If $\alpha$ is disjoint from $W$ then it must be contained within one of the $A_i$. But then one complementary region of $\alpha$ has area less than the area of $A_i$, which in turn is less than $\varepsilon$, so $\alpha$ is clearly not $\varepsilon$-balanced, a contradiction.

If $\alpha$ intersects $W$ but with only inessential or peripheral components, then there are two cases to consider.

The first is that $\alpha\cap W$ is a simple closed curve, and therefore is contained within $W$ but is peripheral in $W$. In this case, one of the complementary regions of $\alpha$ contains at most one of the $A_i$, but then this complementary region has area at most $\mathrm{area}(A_{i})+\mathrm{area}(W)<\varepsilon$, contradicting $\alpha$ is $\varepsilon$-balanced.

The second and final case is that $\alpha\cap W$ is a non-empty union of (possibly infinitely many) connected components. We suppose that each component of $\alpha\cap W$ that intersects the interior of $W$, i.e. an interval in $W$ with endpoints in $\partial W$, contains no essential properly embedded interval in $W$. In this case, we will show that $\alpha$ cannot be $\varepsilon$-balanced. Let's assume without loss of generality that $\alpha$ intersects the open disk $A_1$, and therefore $\alpha$ is disjoint from $A_2,\dots,A_n$ 
(otherwise we can find an essential properly embedded interval in $\alpha\cap W$). We then observe $A_2, \dots, A_n$ lie in the same complementary component of $\overline{A_1} \cup \alpha$: When $\alpha\cap W$ has finitely many components, then this is immediate, however, there may be infinitely many components. Fortunately, we may take arbitrarily small $C^0$ perturbations of $\alpha$, which intersect $\partial W$ transversely and therefore finitely many times, and deduce the same for this family of curves. Since the property of having (or not having) $A_2, \dots, A_n$ lying in the same complementary component of $\overline{A_1} \cup \alpha$ is a $C^0$-open condition on $\alpha$, we deduce that the property holds for $\alpha$.

Finally, since $A_2, \dots, A_n$ lie in the same complementary component of $\alpha$, it must be the case that this complementary component of $\alpha$ has area greater than $1-\varepsilon$, and so $\alpha$ is not $\varepsilon$-balanced.
\end{proof}

It is not difficult to find examples that satisfy the hypotheses of Lemma~\ref{Sufficient condition for W to be a witness}.

\begin{example}\label{Witness example}
Given any $\varepsilon$ with $0<\varepsilon\leq \frac{1}{2}$, let $n$ be an integer such that $n\geq 4$.
Let each of $A_1,\dots ,A_n$ have area equal to $\frac{1}{n+1}$. Then the complement $W$ of the disks is homeomorphic to $S_{0,n}$, satisfying $\mathrm{area}(W)=\frac{1}{n+1}$. Therefore, we have $\mathrm{area}(A_{i})+\mathrm{area}(W)=\frac{2}{n+1}$, which is less than $\varepsilon$ for sufficiently large $n$.
\end{example}

Let $W$ be a witness as described in Lemma~\ref{Sufficient condition for W to be a witness}. In what follows, we write $\mathcal{P}(\mathcal{C}(W))$ for the set of subsets of vertices of $\mathcal{C}(W)$. Then we define the subsurface projection $\pi_{W}$ from $\mathcal{C}^{\dagger}_{\varepsilon}(S^{2})$ to $\mathcal{P}(\mathcal{C}(W))$, by following Masur-Minsky \cite{masur2000geometry}. Given an $\varepsilon$-balanced curve $\alpha$ in $S^2$, since $W$ is a witness, either $\alpha$ represents a curve of $W$ (so set $\pi_W(\alpha)$ to be the set whose element is the isotopy class in $W$ of this curve) or failing this, take the set of (isotopy classes in $W$ of) essential properly embedded intervals of $\alpha \cap W$ in $W$. These are clearly pairwise disjoint and so define a clique of $\mathcal{A}(W)$. Following Masur--Minsky, we may surger any of these arcs to obtain a curve of $W$, and we write $\pi_W(\alpha)$ to be the set of such curves.
It is shown in \cite{masur2000geometry} that $\pi_W(\alpha)$ has diameter at most $2$ in $\mathcal{C}(W)$. Moreover, it is a straightforward consequence of \cite{masur2000geometry} that this defines a coarsely Lipschitz map. This is summarised in the following lemma:

\begin{lemma}\label{coarsely lipschitz}
The subsurface projection $\pi_{W}\colon \mathcal{C}^{\dagger}_{\varepsilon}(S^{2})\rightarrow \mathcal{P}(\mathcal{C}(W))$ is a coarsely Lipschitz map in the following sense. Namely, for any $\alpha,\beta \in \mathcal{C}^{\dagger}_{\varepsilon}(S^{2})$  we have the following:
\begin{enumerate}
    \item $\mathrm{diam}(\pi_W(\alpha))\leq 2$, and
    \item whenever $d_{\mathcal{C}^{\dagger}_{\varepsilon}(S^{2})}(\alpha,\beta)=1$ then $\mathrm{diam}_{\mathcal{C}(W)}(\pi_W(\alpha)\cup\pi_W(\beta))\leq 4$.
\end{enumerate}
\end{lemma}

\begin{proof}
For the proof of $1$, we note that either $\alpha\cap W$ is already a curve, in which case the statement is immediate, or, $\alpha\cap W$ is not a curve but since $W$ is a witness, $\alpha\cap W$ contains essential properly embedded intervals of $W$. These are pairwise disjoint, and so the arguments in \cite[Section~2.3]{masur2000geometry} apply, to give the result.

For the proof of $2$, let $\alpha,\beta\in \mathcal{C}^{\dagger}_{\varepsilon}(S^2)$ with $d_{\mathcal{C}^{\dagger}_{\varepsilon}(S^{2})}(\alpha,\beta)=1$. Then by definition we know that $\alpha$ and $\beta$ intersect transversely (if at all) and the number of intersection points $\alpha\cap\beta$ is at most two.
If $\alpha$ and $\beta$ are already representatives of curves of $W$, then it is straightforward to show that the curves have distance at most $2$ in $\mathcal{C}(W)$.
If $\alpha$ and $\beta$ are both not curves of $W$, then for any pair of essential properly embedded intervals in $W$ of $\alpha$ and $\beta$, their intersection is at most $2$. It is possible to show that their distance in $\mathcal{A}(W)$ is at most $2$. By taking a path in $\mathcal{A}(W)$  of length at most $2$ between these two arcs, and surgering each vertex to get a curve of $W$, it can be shown that $\pi_W(\alpha)\cup\pi_W(\beta)$ has diameter at most $4$ in $\mathcal{C}(W)$.
Finally, the case where exactly one of $\alpha$ or $\beta$ defines a curve of $W$, follows along similar lines. \end{proof}

The proof of the infinite diameter conclusion of Theorem~\ref{infinite diameter} will now be an immediate consequence of the following lemma.

\begin{lemma}\label{hyperbolics}
Let $W$ be a witness for $\mathcal{C}^{\dagger}_{\varepsilon}(S^{2})$ homeomorphic to $S_{0,n}$ with $n\geq 4$. Then for any pure mapping class $\varphi\in\mathrm{PMCG}(W)$ there exists $f\in \Ham(S^2)$ such that $f(W)=W$, and $f|_W$ represents $\varphi$.

Moreover, if $\varphi$ above is a pseudo-Anosov mapping class, then $f$ acts hyperbolically on $\mathcal{C}^{\dagger}_{\varepsilon}(S^{2})$. In particular, $\mathcal{C}^{\dagger}_{\varepsilon}(S^{2})$ has infinite diameter.
\end{lemma}

\begin{proof}
Any pure mapping class $\varphi \in \mathrm{PMCG}(W)$ may be realized by some $f\in \Ham(S^2)$ such that $f(W)=W$, and $f|_W$ represents $\varphi$.
To see this, note that Artin proved that the pure mapping class group $\mathrm{PMCG}(S_{0,n})$ is generated by Dehn twists, so $\varphi$ is a composition of Dehn twists.
It is enough to show that any Dehn twist is realized by some Hamiltonian diffeomorphism  of $S^2$ supported in the interior of $W$.
Indeed, for any smooth simple closed curve $\gamma$ on the sphere, by the \emph{Lagrangian Neighborhood Theorem} \cite{Weinstein, mcduff2017introduction}, we can pick an arbitrarily small open neighborhood of $\gamma$ whose symplectic structure agrees with a small open neighborhood about the core curve of $S^1\times (-\eta,\eta)$ where $\eta \ll 1$.
Now since the latter has symplectomorphisms that represent Dehn twists, we conclude the same about a small open neighborhood of $\gamma$. Now any symplectomorphism of $S^2$ is a Hamiltonian diffeomorphism.
Since every Dehn twist in $\mathrm{MCG}(W)$ can be represented by $h|_W$ for some $h\in \Ham(S^2)$ with $h(W)=W$, by taking compositions, it follows that there exists $f\in \Ham(S^2)$ such that $f(W)=W$ and $f|_W$ represents $\varphi$, as claimed.

Let $\alpha\in \mathcal{C}^{\dagger}_{\varepsilon}(S^{2})$. Suppose that $\varphi$ is a pseudo-Anosov mapping class. According to Proposition \ref{pseudoAnosove mapping classes are loxodromic elements}, we know that for all $n\in \mathbb{N^+}$, we have
$$
\mathrm{diam}_{\mathcal{C}(W)}(\pi_{W}(f^{n}(\alpha))\cup\pi_W(\alpha))=\mathrm{diam}_{\mathcal{C}(W)}(\varphi^n(\pi_W(\alpha))\cup \pi_W(\alpha))\geq Cn 
$$
for some constant $C>0$. Combining this with the fact that $\pi_W$ is $4$-Lipschitz (Lemma~\ref{coarsely lipschitz}) we have that $d_{\mathcal{C}^{\dagger}_{\varepsilon}(S^{2})}(f^n(\alpha),\alpha)\geq  \lfloor \frac{Cn}{4} \rfloor$.
By replacing $f$ with a higher power of itself if necessary, we can ensure that $C\geq 4$, and therefore $d_{\mathcal{C}^{\dagger}_{\varepsilon}(S^{2})}(f^n(\alpha),\alpha)\geq \lfloor \frac{Cn}{4} \rfloor \geq n$. This shows that $f$ is a hyperbolic isometry and the infinite diameter statement follows.\end{proof}

\subsection{Questions related to quasi-isometries} For $0<\varepsilon_1<\varepsilon_2\leq \frac{1}{2}$, we have an inclusion map of $\mathcal{C}^{\dagger}_{\varepsilon_2}(S^{2})$ into $\mathcal{C}^{\dagger}_{\varepsilon_1}(S^{2})$ because every $\varepsilon_2$-balanced curve is $\varepsilon_1$-balanced, and the notion of adjacency is the same. This is a straightforward consequence of the definition. This inclusion is therefore a $1$-Lipschitz embedding. However, it is not a quasi-isometric embedding.

\begin{theorem}\label{inclusion map is not a qi embedding}
For $0<\varepsilon_{1}<\varepsilon_{2}\leq \frac{1}{2}$, the inclusion map $i\colon \mathcal{C}^{\dagger}_{\varepsilon_2}(S^{2})\rightarrow \mathcal{C}^{\dagger}_{\varepsilon_1}(S^{2})$ is $1$-Lipschitz but not a quasi-isometric embedding.
\end{theorem}

We will prove this theorem right after we prove the following lemma.

\begin{lemma}\label{Witness for the big one but not for the small}
Let $0<\varepsilon_1<\varepsilon_2\leq \frac{1}{2}$.
There exists a witness $W\subset S^2$ for $\mathcal{C}_{\varepsilon_{2}}^{\dagger}(S^{2})$ such that a boundary component of $W$ is $\varepsilon_1$-balanced.

\end{lemma} 

\begin{proof} We know that a subsurface $W\subset S^2$ is given by removing $n$ open disks from $S^2$, labelled by $A_1,\dots ,A_n$. Then $W$ will be homeomorphic to $S_{0,n}$. We would like to use Lemma~\ref{Sufficient condition for W to be a witness}, which states that if  for every $i$ we have that $\mathrm{area}(A_{i})+\mathrm{area}(W)<\varepsilon_{2}$, then $W$ is a witness for $\mathcal{C}_{\varepsilon_{2}}^{\dagger}(S^{2})$.
To construct $W$, it is enough first to pick $A_1$ such that $\mathrm{area}(A_1)=\varepsilon_1$. And then pick $A_2,\dots,A_n,W$ each of area \[\frac{1-\mathrm{area}(A_1)}{n}=\frac{1-\varepsilon_1}{n},\] and then pick a sufficiently large $n$ such that we have $\mathrm{area}(A_{i})+\mathrm{area}(W)<\varepsilon_{2}$ for each $i$. 
This is possible because $$\mathrm{area}(A_1)+\mathrm{area}(W)=\varepsilon_1+\frac{1-\varepsilon_1}{n},$$ and since $\varepsilon_1<\varepsilon_2$ this will be less than $\varepsilon_2$ for sufficiently large $n$. For $2\leq i \leq n$ we have $$\mathrm{area}(A_i)+\mathrm{area}(W)=\frac{2-2\mathrm{area}(A_1)}{n},$$ which tends to $0$ as $n$ tends to infinity, so will be less than $\varepsilon_2$ for sufficiently large $n$. This gives us the required witness subsurface $W$.
\end{proof}

\begin{proof}[Proof of Theorem~\ref{inclusion map is not a qi embedding}]
For all $\alpha,\beta\in \mathcal{C}^{\dagger}_{\varepsilon_2}(S^{2})$, we have
$$
d_{\mathcal{C}^{\dagger}_{\varepsilon_1}(S^{2})}(\alpha,\beta)\leq d_{\mathcal{C}^{\dagger}_{\varepsilon_2}(S^{2})}(\alpha,\beta),
$$ which is immediate from the definitions.
Therefore, the inclusion map from $\mathcal{C}^{\dagger}_{\varepsilon_2}(S^{2})$ to $\mathcal{C}^{\dagger}_{\varepsilon_1}(S^{2})$ is $1$-Lipschitz.
By Lemma \ref{Witness for the big one but not for the small}, we can find a witness $W$ for $\mathcal{C}^{\dagger}_{\varepsilon_2}(S^{2})$ with a boundary component $\eta$ in $\mathcal{C}^{\dagger}_{\varepsilon_1}(S^{2})$. 
We can pick an $\varepsilon_{2}$-balanced curve $\alpha$ which is disjoint from $\eta$ (and the disk of area $\varepsilon_1$ that it bounds).
Following Lemma~\ref{hyperbolics}, we know that we can find $f\in \Ham(S^2)$ such that $d_{\mathcal{C}^{\dagger}_{\varepsilon_{2}}(S^{2})}(f^n(\alpha),\alpha)\geq n$ for all $n\in \mathbb{N}^{+}$ and such that $f(W)=W$.
However, $d_{\mathcal{C}^{\dagger}_{\varepsilon_{1}}(S^{2})}(f^n(\alpha),\alpha)\leq 2$ for all $n\in \mathbb{N}^{+}$, because $\eta\subseteq \partial W$ is disjoint from both $f^n(\alpha)$ and $\alpha$. By considering sufficiently large values of $n$, we see that the inclusion map cannot be a quasi-isometric embedding.
\end{proof}

\begin{question*}
For $0<\varepsilon_1<\varepsilon_2\leq \frac{1}{2}$, one can ask whether the corresponding balanced curve graphs $\mathcal{C}^{\dagger}_{\varepsilon_1}(S^{2})$ and $\mathcal{C}^{\dagger}_{\varepsilon_2}(S^{2})$ are quasi-isometric to each other, or have quasi-isometric embeddings between them. From the above discussion, the inclusion map cannot be a candidate, but this does not rule out other possible choices of maps.
\end{question*}

\section{Hyperbolicity of the balanced curve graph}\label{section-Hyperbolicity of the balanced curve graph}

In this section, for any $0<\varepsilon\leq \frac{1}{2}$, we prove that the $\varepsilon$-balanced curve graph is connected and Gromov hyperbolic. The fact that the diameter is infinite was established in Lemma~\ref{hyperbolics}.
We rely on an important construction that we introduce in Section~\ref{subsection-projections and the guessing geodesics}: For each pair of vertices $\alpha$ and $\beta$, we construct a subgraph $\mathcal{L}(\alpha,\beta)$ containing them.
The connectivity of these subgraphs is proved in Section~\ref{Subsection-connectivity}. 
We carefully verify that our subgraphs satisfy the hypotheses of the `guessing geodesics lemma' (Lemma~\ref{Guessing Geodesic Lemma}) in Section~\ref{Subsection-hyperbolicity}. Gromov hyperbolicity then follows.

\subsection{Projections and `Guessing Geodesics'}
\label{subsection-projections and the guessing geodesics}

In this section, for every pair of vertices $\alpha$ and $\beta$ in $\mathcal{C}^{\dagger}_{\varepsilon}(S^2)$, we construct a subgraph $\mathcal{L}(\alpha,\beta)$ containing $\alpha$ and $\beta$.
Our goal, established later on in Sections~\ref{Subsection-connectivity} and \ref{Subsection-hyperbolicity}, is to prove that these subgraphs satisfy the criteria in Lemma~\ref{Guessing Geodesic Lemma}.

First we must treat the degenerate case, for which we define $\mathcal{L}(\alpha,\alpha)=\{\alpha\}$. Of course, the most interesting case is when we have $\alpha\neq \beta$. In this case, the pair of curves may or may not be disjoint, or may intersect transversely or not transversely.
We will write $\alpha \pitchfork \beta$ to indicate that $\alpha$ and $\beta$ are disjoint, or intersect transversely.
It is clear that if the curves intersect transversely then the intersection $\alpha\cap\beta$ has finite cardinality.

Now, to define $\mathcal{L}(\alpha,\beta)$ in full generality, we first define $\mathcal{L}_\pitchfork(\alpha,\beta)$ for pairs of curves such that $\alpha\pitchfork\beta$. 
Roughly, the idea is to consider connected subarcs/curves $\alpha'\subseteq \alpha$ and $\beta'\subseteq \beta$, and attempt to ``project'' $\beta'$ to a set of curves $\pi_{\alpha'}(\beta')$ that is disjoint from $\alpha'$. 
By interpolating between smaller choices of $\alpha'$ and larger choices of $\beta'$, we may find a quasi-path from $\alpha$ to $\beta$ by using the sets $\pi_{\alpha'}(\beta')$ (we justify this in Section~\ref{Subsection-connectivity}). 
Let us now fill in the details of the construction by first defining $\pi_{\alpha'}(\beta')$.

\begin{definition}\label{projection}

Let $\alpha$ and $\beta$ be simple closed curves in $S^2$  and $\alpha \pitchfork \beta$. Let $\alpha'\subset \alpha$ and $\beta'\subset \beta$ be connected compact subarcs. 
We write $\pi_{0}(\beta'\backslash \alpha')$ for the set of connected components of $\beta'\backslash \alpha'$.
We define the projection $\pi_{\alpha'}(\beta')$ to be the subset of equators whose elements are $\gamma\in \mathcal{C}^{\dagger}_{\frac{1}{2}}(S^2)$ with the following properties: 

\begin{enumerate}
    \item $\gamma\cap \alpha'=\emptyset$,
    \item $\gamma\pitchfork \beta'$, and $\gamma$ is disjoint from the endpoints of $\beta'$, and
    \item for all connected components $b\in \pi_{0}(\beta'\backslash \alpha')$, we have $|\gamma\cap b|\leq 2$.
\end{enumerate}
Furthermore, we set $\pi_{\alpha}(\beta')=\pi_{\alpha}(\beta)=\{\alpha\}$.\end{definition}

\begin{remark} 
By the definition of $\pi_{\alpha'}(\beta')$, for $\beta''\subset\beta'$ we must have that,
$$
\pi_{\alpha'}(\beta')\subseteq \pi_{\alpha'}(\beta''),
$$ which is a straightforward consequence of the definition. Any equator in $\pi_{\alpha'}(\beta')$ will satisfy the first two conditions of the definition above. The third condition will hold because if an equator intersects each component of $\beta'\setminus \alpha'$ at most twice transversely, then it must also be true for $\beta''\setminus \alpha'$, and so the equator is also an element of $\pi_{\alpha'}(\beta'')$.
\end{remark}

Let $0<\varepsilon\leq \frac{1}{2}$. 
For $\alpha'$ and $\beta'$ as in Definition~\ref{projection}, the set $\pi_{\alpha'}(\beta')$ will be a subset of $\mathcal{C}^{\dagger}_{\varepsilon}(S^2)$.
Ideally, we want this to be a uniformly bounded diameter set in $\mathcal{C}^{\dagger}_{\varepsilon}(S^2)$, if we are to use this to obtain a quasi-path from $\alpha$ to $\beta$.
To begin our analysis on this, we have two important cases to consider:
\begin{enumerate}
    \item There is a connected component of $S^2\setminus (\alpha^{\prime}\cup \beta^{\prime})$ whose area is strictly greater than $1-\varepsilon$. In this case, we say that $\alpha'$ and $\beta'$ are \emph{inadmissible} for $\varepsilon$. 
    \item Each connected component of $S^2\setminus(\alpha'\cup \beta')$ has area less than or equal to $1-\varepsilon$. In this case, we say that $\alpha'$ and $\beta'$ are \emph{admissible} for $\varepsilon$.
\end{enumerate}

We are now ready to prove

\begin{lemma}
    Let $0<\varepsilon\leq \frac{1}{2}$ and let $\alpha'$ and $\beta'$ be as in Definition~\ref{projection}. If $\alpha'$ and $\beta'$ are inadmissible for $\varepsilon$, then $\pi_{\alpha'}(\beta')$ has infinite diameter in $\mathcal{C}^{\dagger}_{\varepsilon}(S^2)$.
\end{lemma}

\begin{proof}
    There is a connected component $X$ of $S^2\setminus(\alpha'\cup\beta')$ whose area is strictly greater than $1-\varepsilon$.
    Now consider a closed disk $X'\subset X$ whose area is also strictly greater than $1-\varepsilon$. 
    The complement of $X'$ has area strictly less than $\varepsilon$.
    By the construction of the witness/witnesses $W$ in Lemma~\ref{Witness for the big one but not for the small}, by taking $\varepsilon_1=\mathrm{area}(S^2\setminus X')$ and $\varepsilon_2=\varepsilon$, there is a witness $W$ for $\mathcal{C}^\dagger_\varepsilon(S^2)$ with $\partial X'\subset \partial W$, and $W\subset X'$.
    Now, $X'$ contains an equator $\gamma$. Let $f$ be a Hamiltonian diffeomorphism as in Lemma~\ref{hyperbolics}, which is supported on $W$ (and therefore $X'$) and acts hyperbolically. Then for every $i\in\mathbb{Z}$, we have that $f^i\gamma\subset X'$. It is then immediate by the definition that every $f^i\gamma$ is contained in $\pi_{\alpha'}(\beta')$ because $\gamma\subset X'$ is disjoint from $\alpha'\cup\beta'$. Since the distance between $\gamma$ and $f^i\gamma$ can be arbitrarily large (as shown in the proof of Theorem~\ref{inclusion map is not a qi embedding}) we deduce that $\pi_{\alpha'}(\beta')$ has infinite diameter in $\mathcal{C}^\dagger_\varepsilon(S^2)$ as required.\end{proof}

Conversely, we will show later that if $\alpha'$ and $\beta'$ are admissible for $\varepsilon$ then $\pi_{\alpha'}(\beta')$ will have finite diameter, bounded by a universal constant, see Lemma~\ref{projection gives you uniformly bounded set}. Suppose $\alpha$ and $\beta$ are transverse. By varying $\alpha'$ and $\beta'$ which are admissible for $\varepsilon$, and considering $\pi_{\alpha'}(\beta')$ and $\pi_{\beta'}(\alpha')$, we will obtain many curves that in some sense interpolate between $\alpha$ and $\beta$ in $\mathcal{C}^\dagger_\varepsilon(S^2)$. A key definition is as follows.

\begin{definition}\label{connected subgraph when curves intersect transversely}
Let $\alpha,\beta\in \mathcal{C}^{\dagger}_{\varepsilon}(S^{2})$. Suppose that $\alpha \pitchfork \beta$. We define 
$$
\mathcal{L}_{\pitchfork}(\alpha,\beta)\coloneqq\bigcup_{\alpha',\beta'} \pi_{\alpha'}(\beta')\cup\pi_{\beta'}(\alpha'),
$$
where the union above is taken over all $\alpha'\subseteq \alpha$ and $\beta'\subseteq \beta$ such that

\begin{enumerate}
    \item $\alpha'$ and $\beta'$ are admissible for $\varepsilon$,
    \item $\alpha'$ is a connected subarc of $\alpha$ or $\alpha'=\alpha$, and
    \item $\beta'$ is a connected subarc of $\beta$ or $\beta'=\beta$.
\end{enumerate}

\end{definition}

\begin{remark}
    We always have that $\mathcal{L}_{\pitchfork}(\alpha,\beta)$ is non-empty. Indeed, observe that $\alpha$ and $\beta$ will always be an element of $\mathcal{L}_{\pitchfork}(\alpha,\beta)$ because we may pick $\alpha'=\alpha$ and $\beta'=\beta$, in which case $\alpha'$ and $\beta'$ satisfy the conditions in the definition above. We then have, for example, that $\pi_{\alpha'}(\beta')=\{\alpha\}$, so $\alpha$ is in $\mathcal{L}_{\pitchfork}(\alpha,\beta)$ (and similarly for $\beta$).
\end{remark}

A key lemma is the following 

\begin{lemma}\label{coarsely connected}
For every $\alpha,\beta\in \mathcal{C}^{\dagger}_{\varepsilon}(S^{2})$ with $\alpha\pitchfork \beta$, the $8$-neighbourhood of $\mathcal{L}_{\pitchfork}(\alpha,\beta)$ is connected.
\end{lemma}

We will prove Lemma~\ref{coarsely connected} in Section~\ref{Subsection-connectivity} using the methods in Section~\ref{Subsection-Key Lemmas}.

We now define our `guess' $\mathcal{L}(\alpha,\beta)$ for a geodesic (or efficient path) between $\alpha$ and $\beta$ in general.

\begin{definition}
Let $0<\varepsilon\leq \frac{1}{2}$ and let $\alpha,\beta\in \mathcal{C}^{\dagger}_{\varepsilon}(S^{2})$. We define $\mathcal{L}(\alpha,\alpha)=\{\alpha\}$. Now suppose instead that $\alpha\neq \beta$. Write $P=P(\alpha,\beta)$ for the set of pairs $(\alpha_0,\beta_0)$ such that $\alpha_0,\beta_0\in \mathcal{C}^{\dagger}_{\varepsilon}(S^{2})$, $\alpha$ is adjacent to $\alpha_0$, $\beta$ is adjacent to $\beta_0$, and $\alpha_0\pitchfork\beta_0$. 
We define 
$$
\mathcal{L}(\alpha,\beta)\coloneqq  N_8 \left( \bigcup_{(\alpha_0,\beta_0)\in P} \mathcal{L}_{\pitchfork}(\alpha_0,\beta_0) \right).
$$
By Lemma~\ref{coarsely connected}, the subgraphs spanned by $\mathcal{L}(\alpha,\beta)$ are connected.

\end{definition}

\subsection{Criteria to ensure small distances}\label{Subsection-Key Lemmas}
In this section, we provide sufficient criteria for a pair of $\varepsilon$-balanced curves to have small distance in $\mathcal{C}^{\dagger}_{\varepsilon}(S^{2})$, which plays an essential role in the proofs for the rest of the paper.

\begin{definition}
    Let $\alpha,\beta$ be simple closed curves on $S^2$ such that $\alpha\pitchfork\beta$. We say that a complementary component $D$ of $\alpha\cup\beta$ is a \emph{bigon} if $\bar D \setminus D$ is a union of a connected subarc of $\alpha$ and a connected subarc of $\beta$. In general, the number of connected subarcs of $\alpha$ and $\beta$ that make up $\bar D\setminus D$ is the number of \emph{sides} of the complementary component $D$, so $D$ is a bigon if and only if its number of sides is equal to $2$.
\end{definition}

\begin{lemma}\label{trivial intersections}
Let $\alpha,\beta\in \mathcal{C}^{\dagger}_{\varepsilon}(S^{2})$. Assume that $\alpha\pitchfork \beta$ and that there are at most two complementary components of $\alpha\cup \beta$ that are not bigons. Then there is an equator adjacent to both $\alpha$ and $\beta$ and so $d_{\mathcal{C}^{\dagger}_{\varepsilon}(S^{2})}(\alpha,\beta)\leq 2$.

\end{lemma}

\begin{proof}

The lemma is immediate if $\alpha$ and $\beta$ are disjoint, so from now on, suppose that they intersect.
Let us write $I=|\alpha\cap\beta|$. Because $S^2\setminus\alpha$ has two components, we have that $I>0$ is even.
If $I=2$ the conclusion of the lemma is immediate by choosing another equator in a small neighbourhood of $\alpha$, that maintains exactly two transverse intersections with both $\alpha$ and $\beta$. 
From now on we assume that $I\geq 4$.

We consider an arbitrary connected component $D$ of $S^2\setminus \alpha$, which is homeomorphic to an open disk by the Jordan--Schoenflies theorem.
Each connected component of $\beta\setminus \alpha$ is an (open) arc whose closure is a closed arc with two distinct endpoints in $\alpha$.
There is a tree $T=T_{\alpha,\beta}$, dual to $\beta$, whose vertices are the connected components of $D\setminus \beta$, and whose edges span $R_1$ and $R_2$ (distinct components of $D\setminus\beta$) if there is a connected component $b\subset \beta\setminus \alpha$ such that $b\subset \bar R_1\cap \bar R_2$. 
 Then the number of edges of $T$ is equal to $\frac{I}{2}$.

The degree of a vertex of the dual tree $T$ is one half of its number of sides of the corresponding complementary component of $\alpha\cup\beta$. 
The \emph{leaves} of $T$ (the vertices of degree $1$) therefore correspond to the bigons of $\alpha\cup\beta$ that are contained in $D$. 
In particular, the assumption that there are at most two non-bigon regions implies that $T$ has at most two non-leaves. 
When $I>2$, the number of edges of $T$ is at least $3$, in which case it is immediate that $T$ will have at least $1$ non-leaf. 
But there are two dual trees: one for each complementary component of $\alpha$. We conclude that each dual tree has at most one non-leaf when $I>2$.

When $I>2$, it is therefore the case that $T$ is exactly the graph where there is exactly one vertex of degree $\frac{I}{2}$, and all other vertices are leaves. 
(As an aside, when $I=2$, $T$ is precisely the tree with two vertices and one edge.)

Now, we are able to deduce $\alpha\cup\beta$ up to homeomorphism on $S^2$. 
Indeed, knowing the isomorphism type of $T$ determines $\beta\cap D$ up to homeomorphisms of $D$, for any complementary component $D$ of $\alpha$.
However, there are two complementary components $D_1$ and $D_2$ of $\alpha$, and how they glue together can affect the homeomorphism type of $\alpha\cup\beta$.
Fortunately, we know that $\beta$ is connected, which limits the possible ways in which we glue together the disks $D_1$ and $D_2$.
In any case, we directly observe that $\alpha\cup\beta$ is the same up to homeomorphism of $S^2$.
An example of this for $I=10$ is given in Figure~\ref{fig: Two balanced curves with only two non-bigon complementary regions}.

\begin{figure}[H]
    \centering
    \includegraphics[width=0.33\textwidth]{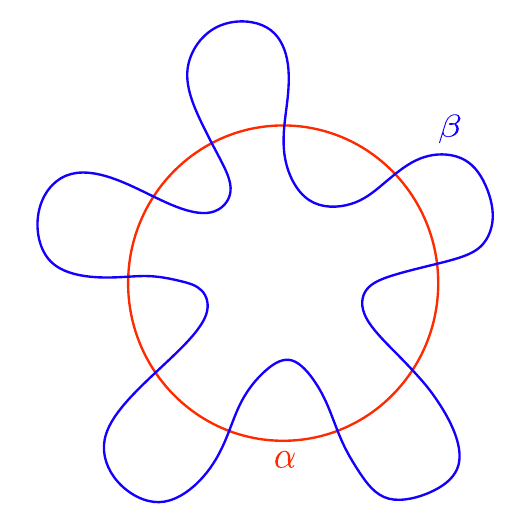}
    \caption{Two $\varepsilon$-balanced curves with only two non-bigon complementary regions}
    \label{fig: Two balanced curves with only two non-bigon complementary regions}
\end{figure}
If we orient $\alpha$, then we obtain an ordering (not necessarily unique!) of the bigons of $\alpha\cup\beta$, namely $B_1, B_2, \dots, B_I$, by the sequence of subarcs of $\alpha$ that the $\bar B_i$ contain.
Pick the complementary component $D$ of $S^2\setminus\alpha$ that contains $B_2$ not $B_1$.
Then $D$ contains $B_2, B_4, \dots, B_I$, and write $D'$ for the non-bigon region in $D$.

For every $j$ such that $1\leq j \leq I$, we may consider a closed disk $N_j$ that contains $B_1\cup\dots\cup B_j$.
We can choose $N_j$ to have an area which is arbitrarily close (but not equal) to the measure of $B_1\cup\dots\cup B_j$. 
Moreover, we can ensure that $N_k\subset N_{k+1}$.
We define $N_{I+1}$ a little differently, namely, we choose $N_{I+1}$ to be such that $N_I\subset N_{I+1}$, but we can arrange $N_{I+1}$ to contain an arbitrary amount of the measure of $D'$ (without containing all of it), see Figure~\ref{fig: The construction of the balanced curve}(c).
\begin{figure}[H]
    \centering
    \includegraphics[width=1.0\textwidth]{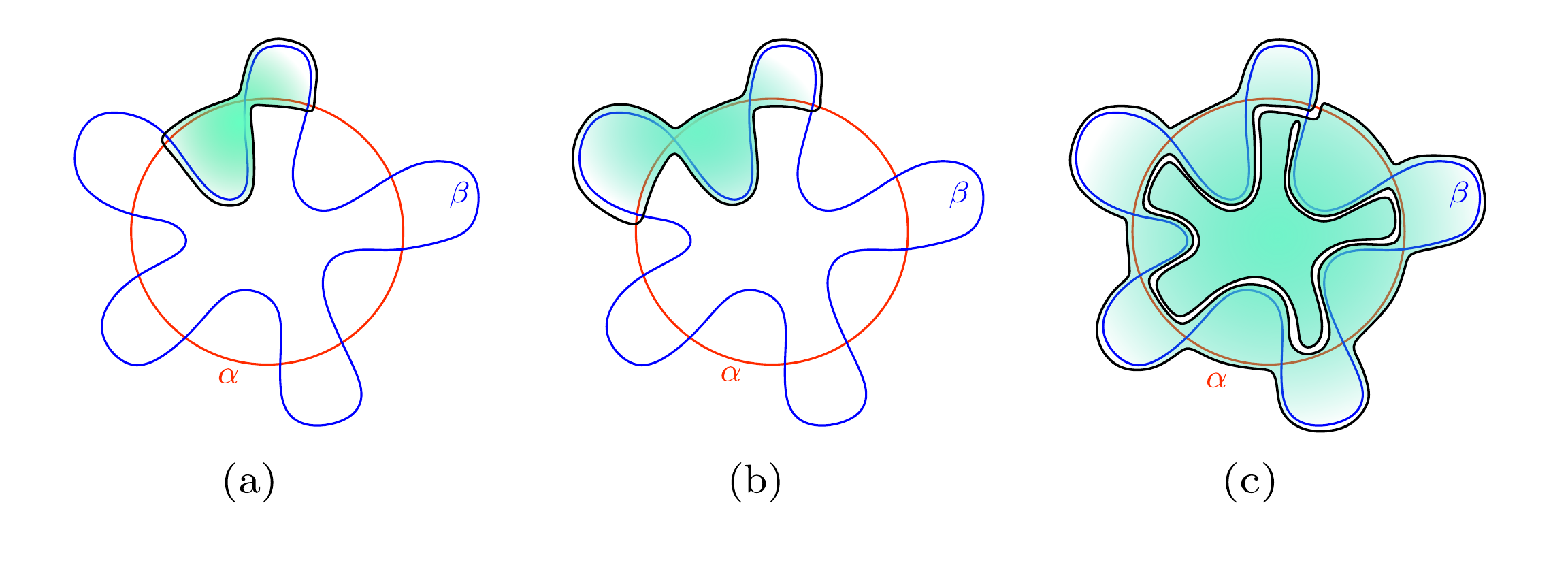}
    \caption{The construction of the $\varepsilon$-balanced curve}
    \label{fig: The construction of the balanced curve}
\end{figure}

Now, since $I\geq 4$, we could have chosen $B_1$ to have been a smallest area bigon to begin with, and therefore its area can be assumed to be strictly less than $\frac{1}{2}$, and so we may arrange $N_1$ to have area strictly less than $\frac{1}{2}$.
We can also ensure that $N_{I+1}$ has area greater than $1-\varepsilon\geq \frac{1}{2}$, because the remaining complementary component of $\alpha\cup\beta$ has area at most $\varepsilon$, see Figure~\ref{fig: The construction of the balanced curve}(c).

Since $N_1\subset\dots\subset N_{I+1}$ is a nested sequence of closed disks on $S^2$, it follows that $\partial N_1$ is isotopic to $\partial N_{I+1}$. Moreover, throughout such an isotopy, we can arrange the curves to have exactly two transverse intersections with both $\alpha$ and $\beta$ respectively. Since this process is continuous, at some point we obtain an equator, while maintaining exactly two transverse intersections with $\alpha$ and $\beta$, and the lemma is proved.
\end{proof}

\begin{lemma}\label{We can straighten the arc in side R}
Let $R\subset S^2$ be a closed topological disk with smooth boundary, and write $\partial R$ for its boundary. Suppose that $\mathrm{area}(R)\leq 1-\varepsilon$. Let $\gamma$ be an $\varepsilon$-balanced curve. Assume that $\gamma$ intersects $\partial R$ transversely with only two non-bigon complementary regions. Let $f\in \Ham(S^2)$ be supported on $R$. Then $d_{\mathcal{C}^{\dagger}_{\varepsilon}(S^{2})}(\gamma,f\gamma)\leq 4$.

\end{lemma}
\begin{proof}

We consider two cases.
\begin{itemize}
    \item Case (i): If $\varepsilon\leq\mathrm{area}(R)\leq 1-\varepsilon$, then we consider $\eta=\partial R$, which is $\varepsilon$-balanced and intersects each of $\gamma$ and $f\gamma$ with only two non-bigon regions.
    By Lemma \ref{trivial intersections} we know that $d_{\mathcal{C}^{\dagger}_{\varepsilon}(S^{2})}(\gamma,\eta)\leq 2$ and $d_{\mathcal{C}^{\dagger}_{\varepsilon}(S^{2})}(f\gamma,\eta)\leq 2$. Then we have $d_{\mathcal{C}^{\dagger}_{\varepsilon}(S^{2})}(\gamma,f\gamma)\leq 4$ by the triangle inequality.
    
    \item Case (ii): $\mathrm{area}(R) < \varepsilon$. It's enough to argue that there is an equator $\eta$ in the complement of the interior of $R$, which intersects $\gamma$ with at most two non-bigon regions. Indeed, starting with the curve $\partial R$, we may enlarge $R$ via ambient isotopies that do not preserve area but fixes each point of $\partial R\cap \gamma$ throughout the isotopy. Throughout this isotopy, the homeomorphism class of $\partial R\cup \gamma$ does not change, but the area of $R$ increases, and eventually becomes $\frac{1}{2}$, giving the required equator. Hence, as we did before, by Lemma \ref{trivial intersections}, we get $d_{\mathcal{C}^{\dagger}_{\varepsilon}(S^{2})}(\gamma,f\gamma)\leq 4$.
\end{itemize}
\end{proof}

\subsection{Proof of connectivity}\label{Subsection-connectivity}
In this section, we will prove the connectivity of $\mathcal{C}^{\dagger}_{\varepsilon}(S^{2})$, for any choice of $\varepsilon$ such that $0<\varepsilon\leq \frac{1}{2}$.

The first step is to ensure that $\pi_{\alpha'}(\beta')$ has finite diameter with a universal upper bound, where $\alpha'\subset \alpha$ and $\beta'\subseteq \beta$ are admissible for $\varepsilon$.

To do this, we start by establishing some structure.
Recall the dual tree $T=T_{\alpha',\beta'}$ for the pair $(\alpha',\beta')$.
We will endow almost all of the edges of the dual tree with a direction as follows.
We note that by definition, there is a one-to-one correspondence between the edges of $T$ and the connected components of $\beta'\setminus \alpha'$ which separate $S^2\setminus\alpha'$ into two.
Each edge $e\in E(T_{\alpha',\beta'})$ gives a partition of $S^2\setminus \alpha'=B_1\cup B_2$, into two disks $B_1$ and $B_2$.
Let $e$ connect the two vertices $v_1$ and $v_2$ such that the region corresponding to $v_i$ is contained in $B_i$.
We assign a direction to $e$ as follows:
There is an arrow from $v_1$ to $v_2$ if and only if $\mathrm{area}(B_1)<\mathrm{area}(B_2)$.

Only at most one edge does not get a direction, because the components $\beta'\setminus\alpha'$ are disjoint and have a non-zero amount of area between them. In the case where there is such an edge, endow an arbitrary direction to this edge. We now have a directed dual tree.

We claim that the directions on $T$ are equal to the following: that there is a choice of vertex $v$, such that every edge $e$ of $T$ points to $v$.
To see why this is true, pick an arbitrary edge $e$ of $T$. 
Then $e$ has an initial vertex $w$ and a terminal vertex $w'$, whose corresponding regions are contained in $B$ and $B'$ respectively, where $B$ and $B'$ are the open disks of $S^2\setminus\alpha'$ in the complement of $e$.
Because the area of $B'$ is at least $\frac{1}{2}$, it is immediate from the definition of the directions of $T$ that every edge in the component of $T\setminus e$ corresponding to $B$ is directed towards $w'$ (and hence $w$).
In particular, there is at most one edge pointing away from $w$, namely $e$, because all the other edges are in $B$.
Therefore every vertex $w$ of $T$ has at most edge pointing away from it.
Finally, the choice of vertex $v$ (as mentioned at the start of this paragraph) is determined by the terminal vertex of any maximal directed embedded path in $T$.
It is clear that $v$ exists.
Since $T$ is connected, this choice of $v$ must be unique.
We call $v$ the \emph{central} vertex of $T$.

We are now ready to prove

\begin{lemma}\label{projection gives you uniformly bounded set}
If $\alpha'$ and $\beta'$ are admissible for $\varepsilon$, then $\pi_{\alpha'}(\beta')$ is a nonempty set and $\mathrm{diam}(\pi_{\alpha'}(\beta'))\leq 8$.

\end{lemma}

\begin{proof}
To see why it is nonempty, there is an embedding of the dual tree $T_{\alpha',\beta'}$ where each vertex is mapped to its corresponding complementary region of $\alpha'\cup\beta'$ and each edge crosses $\beta'$ exactly once transversely at the corresponding component of $\beta'\setminus\alpha'$. Therefore, we may choose a closed disk $N$ that contains the image within its interior, but such that $\partial N$ intersects each component of $\beta'\setminus\alpha$ at most twice and transversely. We have that $N$ meets every complementary component, and therefore, by isotoping $N$ away from $\beta'$ to be larger or smaller within in complementary region, we can ensure that $\partial N$ is an equator and therefore $\partial N\in \pi_{\alpha'}(\beta')$.

Now we prove the upper bound on the diameter. Let $I$ be the central vertex of the dual tree $T=T_{\alpha',\beta'}$, and let $D_1,\dots,D_k$ be the connected components of $T\setminus I$.
By definition of the central vertex of $T$, we have $\mathrm{area}(D_i)\leq \frac{1}{2}$ for each $D_i$.
Now write $D_0=I$ and let $0\leq m \leq k$ be maximal such that 
$$\sum_{i=0}^m \mathrm{area}(D_i)\leq 1-\varepsilon.$$
Note that $m$ exists because the above is satisfied for $m=0$ because $\alpha'$ and $\beta'$ are admissible for $\varepsilon$.
Also, since the total area is $1$, which is greater than $1-\varepsilon$, by maximality of $m$ we have 
$$1-\varepsilon < \sum_{i=0}^{m+1}\mathrm{area}(D_i),$$ and in particular we have $m<k$.

Now, if $\sum_{i=0}^m \mathrm{area}(D_i)= 1-\varepsilon,$ then the argument proceeds as follows. 
Given any $\gamma_1,\gamma_2\in\pi_{\alpha'}(\beta')$, we pick a sufficiently small closed neighbourhood $N$ of $\alpha'$ (which is a disk) such that $\partial N$ is disjoint from $\gamma_1$ and $\gamma_2$, and such that $\partial N$ intersects each component of $\beta'\setminus\alpha'$ transversely and at most twice.
Let $b_{m+1},\dots,b_k$ be the components of $\beta'\setminus\alpha'$ that correspond to each $D_i$, this is a nonempty collection of arcs because $m<k$.
We now consider $\overline {D_0\cup\dots\cup D_m}$ which has area $1-\varepsilon$, but this may not be a disk. 
To avoid this issue, we can consider instead $\overline {D_0\cup\dots\cup D_m}\setminus \mathrm{int}N$, which is a closed disk, but it has area strictly less than $1-\varepsilon$ (which will not be sufficient when $\varepsilon=\frac{1}{2}$). 
To fix this, we consider a sufficiently smaller neighbourhood $N'\subset N$ of $\alpha'$. 
The upshot here, is that we consider the closed disk $\overline {D_0\cup\dots\cup D_m}\setminus \mathrm{int}N'$, which can have area strictly less but arbitrarily close to $1-\varepsilon$, but then by making the disk $\overline {D_0\cup\dots\cup D_m}$ larger by adding a small disk of $N\setminus N'$ only, we can make it have area exactly equal to $1-\varepsilon$.
Let us write $D$ for this disk, then $\gamma=\partial D$ is $\varepsilon$-balanced.
Moreover, $\gamma$ can be decomposed into arcs: subarcs of $b_i\cap (S^2\setminus N)$, and subarcs contained only in $N$.
It follows that $\gamma_1$ and $\gamma_2$ only intersect $\gamma$ at the $b_i\cap(S^2\setminus N)$.
Moreover, if $\gamma_i$ intersects $b_j$ then it does so exactly twice, and forms a bigon within the open disk corresponding to $D_j$.
It follows that $\gamma_i$ and $\gamma$ intersect transversely and only have at most two non-bigon complementary regions, so Lemma~\ref{trivial intersections} applies, and we are done.

The remaining case is $\sum_{i=0}^m \mathrm{area}(D_i) < 1-\varepsilon.$
Given any $\gamma_1,\gamma_2\in\pi_{\alpha'}(\beta')$, we wish to use the proof above, but there is not  enough area (for the argument to work for $\varepsilon=\frac{1}{2}$), and so we need to expand our disk $D$ into $D_{m+1}$ to get enough area.
However, a complication that can arise is that there may be arbitrarily many intersections between $\gamma_1$ and $\gamma_2$ inside $D_{m+1}$, which prevents us from completing the proof as above.
To remove the complication, we wish to apply $f$ supported on a disk $R$ with $\mathrm{area}(R)<1-\varepsilon$ so that $\gamma_1$ and $f\gamma_2$ intersect zero times or twice transversely within $D_{m+1}$. 
See Figure~\ref{fig: Case analysis} for a picture.
From here, it is not hard to see that there is a choice of $D$ that can be made so that $\mathrm{area}(D)=1-\varepsilon$ (by expanding $D$ into $D_{m+1}$) and Lemma~\ref{trivial intersections} can be applied to $\gamma_1$ and $\gamma=\partial D$, $\gamma=\partial D$ and $f\gamma_2$. 
Finally, $f\gamma_2$ and $\gamma_2$ are close by Lemma~\ref{We can straighten the arc in side R}.

\begin{figure}[H]
    \centering
    \includegraphics[width=1.0\textwidth]{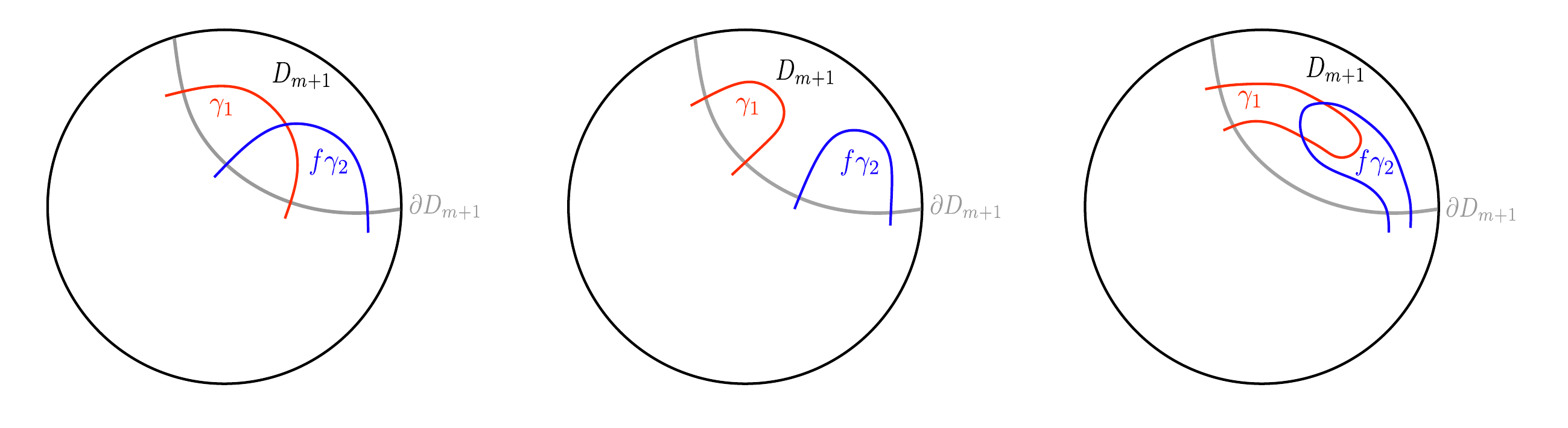}
    \caption{Case analysis}
    \label{fig: Case analysis}
\end{figure}

So to complete the proof, it remains to provide the closed disk $R$, the area-preserving map $f$ supported on $R$, so that $\gamma_1$ and $f\gamma_2$ are disjoint or intersect twice transversely within $D_{m+1}$.
In fact, the order in which we do this is different.
We will describe $R$ and $\gamma$, such that there exists area-preserving $f$ supported on $R$ such that $f\gamma_2=\gamma$.
To do this, we let $R'$ be a closed disk contained in $D_{m+1}$ (recall $D_{m+1}$ is an open set), which contains almost all of the area of $D_{m+1}$. 
We may choose a closed disk $R$ containing $R'$ within its interior, such that the interior of $R$ contains $b_{m+1}\setminus N$, $\gamma_1\cap D_{m+1}$ and $\gamma_2\cap D_{m+1}$, but such that $\mathrm{area}(R)<1-\varepsilon$.
This is possible because $\mathrm{area}(D_{m+1})\leq \frac{1}{2}\leq 1-\varepsilon$.

First we apply an area-preserving $f_1$ supported on $R$ so that $\gamma_1\cap f_1\gamma_2\cap b_{m+1}$ are disjoint, while $f_1\gamma_2\in\pi_{\alpha'}(\beta')$. This can be achieved by perturbing $\gamma_2$ within $R$.

Then we want to replace $ f_1 \gamma_2$ by another equator $\gamma$, such that $\gamma$ and $f_1 \gamma_2$ coincide outside of $R$, and coincide within a neighbourhood of $\partial R$, but such that $\gamma$ and $\gamma_1$ intersect as shown in Figure~\ref{fig: Case analysis}.
To prove this, it suffices to choose positive real numbers to assign to each complementary component of $\gamma_1\cup\gamma$ within $R$, that realise the correct areas of the components of $R\setminus\gamma_1$ and $R\setminus \gamma$ (the latter having to match $R\setminus f_1 \gamma_2$).
So let the areas of the complementary regions in $R$ of $\gamma_1$ be given by $A_1$ and $A_2$, and similarly $A_3$ and $A_4$ for $f_1\gamma_2$.
Write $r$ for the area of $R$.
After relabelling, and abusing notation by writing $A_i=\mathrm{area}(A_i)$, we can assume that $A_1=\frac{r}{2}+t$, $A_2=\frac{r}{2}-t$, $A_3=\frac{r}{2}+s$, and $A_4=\frac{r}{2}-s$, where $0\leq s,t < \frac{r}{2}$.

See Figure~\ref{fig: The choice of gamma} for the choice of $\gamma$, which will have the correct areas for the complementary components.
By picking $\delta>0$ small enough, we can ensure that each component is given a positive real area.
It is well known that $\gamma$ and $f_1 \gamma_2$ will satisfy $f_2f_1\gamma_2=\gamma$ for some choice of area-preserving map $f_2$ supported on $R$.
This completes the construction of the aforementioned $R$ and $f$. We are done by the earlier discussion.
\end{proof}

\begin{figure}[H]
    \centering
    \includegraphics[width=1.0\textwidth]{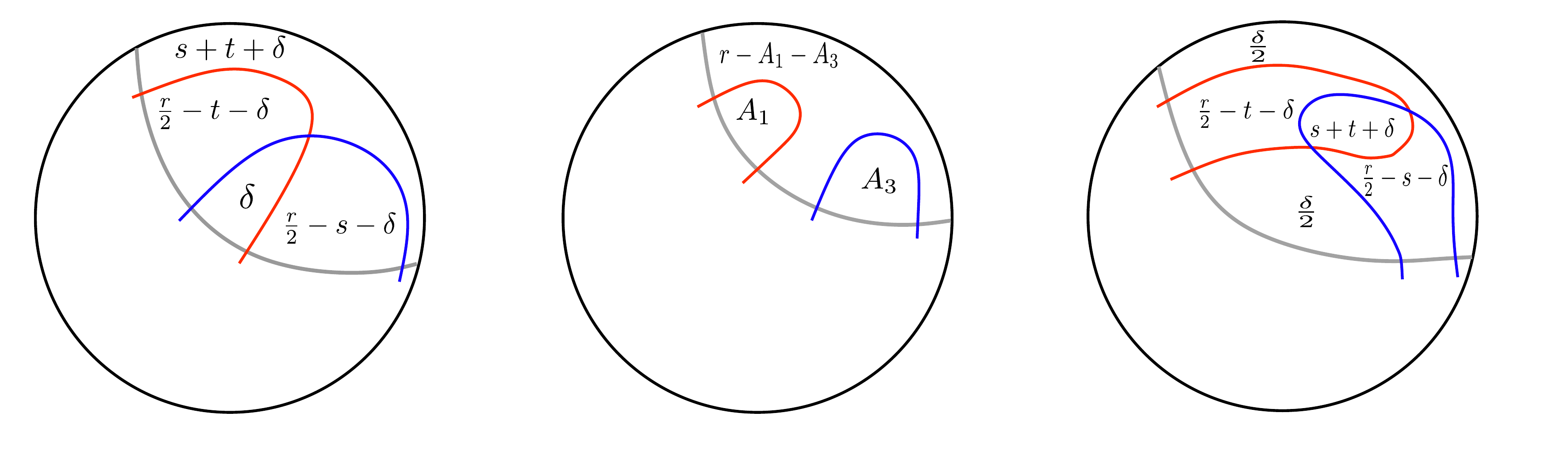}
    \caption{The choice of $\gamma$}
    \label{fig: The choice of gamma}
\end{figure}

The sets $\pi_{\alpha'}(\beta')$ in some sense provide ``quasi-''points of a path between $\alpha$ and $\beta$. The next step is to promote this to a ``quasi-''path. A first step of a quasi-path is the following

\begin{lemma}\label{projection of largest subarc}
Let $\alpha,\beta\in \mathcal{C}^{\dagger}_{\varepsilon}(S^{2})$ such that $\alpha\pitchfork\beta$. Let $\beta'\subseteq \beta$, and let $\alpha'\subset \alpha$ be a connected closed subarc of $\alpha$ such that $|\alpha'\cap\beta'|=|\alpha\cap \beta'|$. Suppose that $\alpha'$ and $\beta'$ are admissible for $\varepsilon$. Then there is an edge from a vertex of $\pi_{\alpha'}(\beta')$ to $\pi_{\alpha}(\beta')=\{\alpha\}$. On the other hand, if $\alpha'$ and $\beta$ are not admissible, then $d_{\mathcal{C}^\dagger_\varepsilon(S^2)}(\alpha,\beta)\leq 2$.

\end{lemma}
\begin{proof}
\begin{figure}[H]
    \centering
    \includegraphics[width=0.5\textwidth]{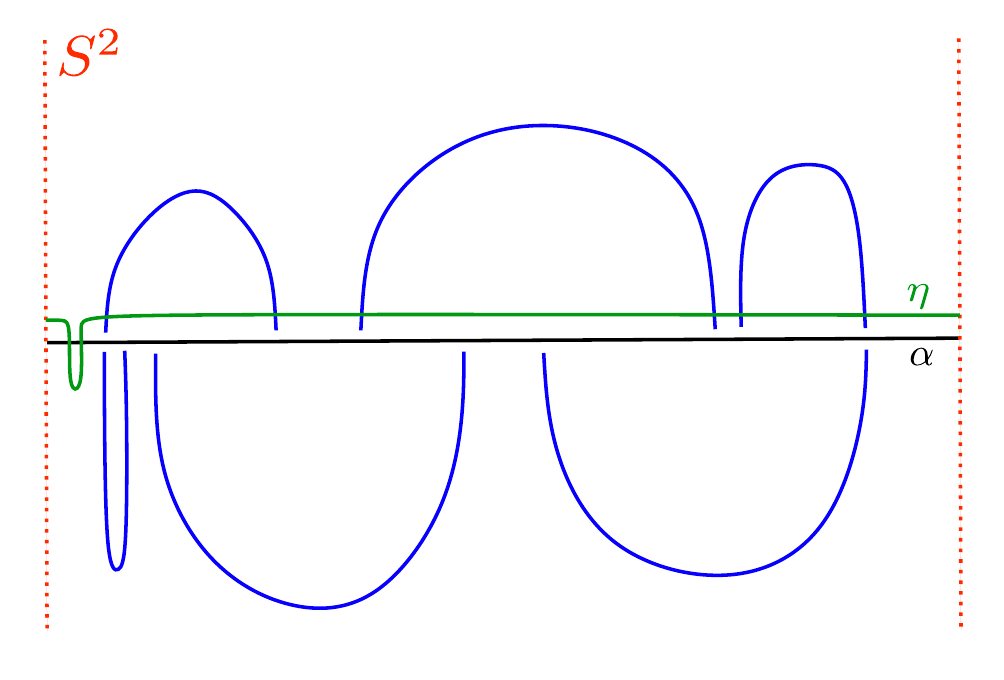}
    \caption{The first step of shrinking $\alpha$}
    \label{fig: The first step of shrinking alpha}
\end{figure}
Figure \ref{fig: The first step of shrinking alpha} gives a schematic picture. The black line is $\alpha$. In the figure, the endpoints of $\alpha$ are identical, since $\alpha$ is a smooth simple closed curve. The red dotted lines are identified as well. We know that $\alpha'\subset \alpha$ is a connected closed subarc of $\alpha$ such that $|\alpha'\cap\beta'|=|\alpha\cap \beta'|$. We can construct an $\varepsilon$-balanced curve $\eta\in \pi_{\alpha'}(\beta')$ such that $\eta \pitchfork \alpha$ and $|\eta\cap \alpha|=2$. Indeed, we can first slightly shift $\alpha$ upwards. 
Since there is a gap between the endpoint and the eliminating intersection point between $\alpha$ and $\beta'$ after shrinking, we can find an $\varepsilon$-balanced curve $\eta$ such that $\eta\pitchfork\alpha$ and $|\eta\cap \alpha|=2$. 
Moreover, $\eta\in \pi_{\alpha'}(\beta')$ since $\eta\cap \alpha' =\emptyset$ and $|\eta\cap b|\leq 2$ for every component $b$ of $\beta' \backslash\alpha'$. 

On the other hand, if $\alpha'$ and $\beta$ are not admissible, then we pick $\eta$ to be an equator within the component of $\alpha'$ and $\beta$ of area larger than $1-\varepsilon$, which intersects $\alpha$ exactly twice transversely along the subarc $\alpha\setminus \alpha'$.
\end{proof}

And now we address the next parts of a quasi-path:

\begin{lemma}\label{shrink beta to get a quasi path}
Let $\alpha,\beta\in \mathcal{C}^{\dagger}_{\varepsilon}(S^{2})$ such that $\alpha\pitchfork\beta$. Let $\beta'\subseteq \beta$, and let $\alpha'' \subset \alpha'$ be connected closed subarcs of $\alpha$ such that 
$$
|\beta'\cap \alpha''|=|\beta'\cap \alpha'|-1,
$$ and $\alpha''$ and $\beta'$ are admissible for $\varepsilon$.
Then $\pi_{\alpha'}(\beta')\cap\pi_{\alpha''}(\beta')\neq\emptyset$.

On the other hand, if $\alpha''$ and $\beta$ are not admissible for $\varepsilon$ but $\alpha'$ and $\beta$ are, then some vertex of $\pi_{\alpha'}(\beta)$ is distance at most $3$ to $\beta$.
\end{lemma}

\begin{proof}
Suppose $\alpha''$ and $\beta$ are admissible for $\varepsilon$, then $\alpha'$ and $\beta$ are also. As shown in Figure \ref{fig: Closed subarcs provide closed projection subsets}, the red dotted line represents the part where $\alpha'$ exceeds $\alpha''$. 
Since $|\beta\cap \alpha''|=|\beta\cap \alpha'|-1$, we know that $\alpha'$ separates some region (whose boundary is a bigon) of $\alpha''\cup\beta'$ into two smaller regions (whose boundaries are also bigons), as shown in Figure \ref{fig: Closed subarcs provide closed projection subsets}. 

\begin{figure}[H]
    \centering
    \includegraphics[width=0.35\textwidth]{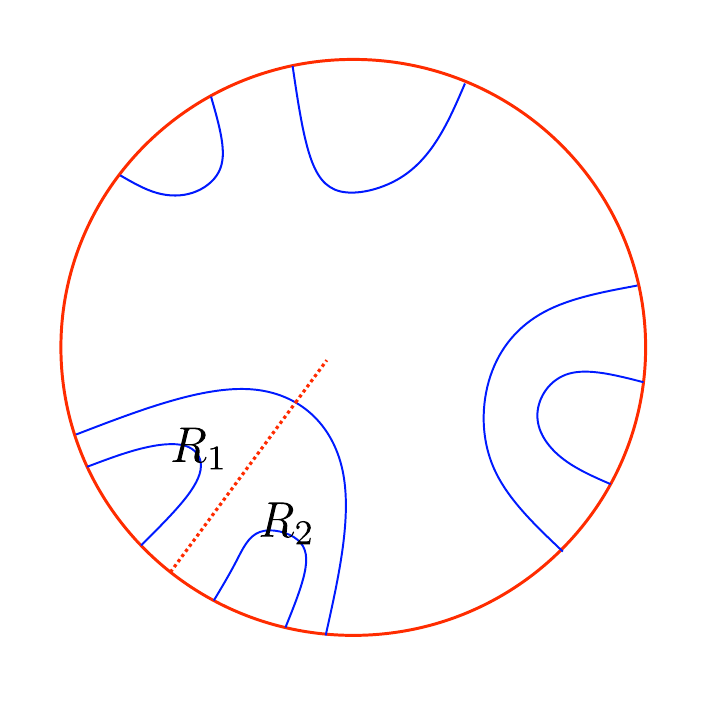}
    \caption{Closed subarcs provide closed projection subsets}
    \label{fig: Closed subarcs provide closed projection subsets}
\end{figure}

We will write $R_1$ and $R_2$ for these smaller regions respectively. 
And we write $R_3$ for the complement of $R_1\cup R_2$. We know that $\mathrm{area}(R_1)+\mathrm{area}(R_2)+\mathrm{area}(R_3)=1$. If $\mathrm{area}(R_1)+\mathrm{area}(R_3)> \frac{1}{2}$, then we can always find an equator in $\pi_{\alpha'}(\beta)$ disjoint from $R_2$ and this equator also belongs to $\pi_{\alpha''}(\beta)$.
Similarly, if $\mathrm{area}(R_2)+\mathrm{area}(R_3)> \frac{1}{2}$, we can deduce that there is an equator in $\pi_{\alpha'}(\beta)$ disjoint from $R_1$ and this equator also belongs to $\pi_{\alpha''}(\beta)$.
In both cases, the lemma is proved.
Now we need to consider the case when both $\mathrm{area}(R_1)+\mathrm{area}(R_3)\leq \frac{1}{2}$ and $\mathrm{area}(R_2)+\mathrm{area}(R_3)\leq \frac{1}{2}$, which implies $\mathrm{area}(R_1)+\mathrm{area}(R_2)+2\mathrm{area}(R_3)\leq 1$, and so $\mathrm{area}(R_3)\leq 0$, which is impossible.

When $\alpha''$ and $\beta$ are not admissible for $\varepsilon$, but $\alpha'$ and $\beta$ are, then it means that two complementary components of $\alpha'\cup\beta$ were glued together to make a complementary component $D$ of $\alpha''\cup\beta$ that has area larger than $1-\varepsilon$.
We then pick an arbitrary equator in the interior of $D$.
Then following the proof that $\pi_{\alpha'}(\beta')\neq\emptyset$ in Lemma~\ref{projection gives you uniformly bounded set}, we can pick the dual tree $T$ to have a single vertex $v$ within the interior of $D$ and $T$ to intersect $\partial D$ only at the edges incident to $v$ in $T$.
This means that the equator $\partial N$ will intersect the equator $\partial D$ with at most two non-bigon regions, and so Lemma~\ref{trivial intersections} applies and we are done.
\end{proof}

Now we can prove the connectivity of $\mathcal{C}^\dagger_{\varepsilon}(S^2)$. We begin with the following which we deduce from the above lemmas.

\begin{proof}[Proof of Lemma~\ref{coarsely connected}]
First, we pick a connected closed subarc $\alpha'$ of $\alpha$, such that $|\alpha'\cap\beta|=|\alpha\cap \beta|$.
It is the case that $\alpha'$ and $\beta$ are admissible for $\varepsilon$, so by Lemma~\ref{projection of largest subarc}, some vertex of $\pi_{\alpha'}(\beta)$ is adjacent to $\alpha$.
Now choose a finite sequence of closed subarcs $\alpha_{i+1}\subset \alpha_i$ of $\alpha_0=\alpha'$ such that $|\alpha_{i+1}\cap\beta|=|\alpha_i\cap \beta|-1$.
It is always the case that $\alpha_i$ and $\beta$ are admissible for $\varepsilon$.
By Lemma \ref{shrink beta to get a quasi path}, we know that $\pi_{\alpha_{i+1}}(\beta)$ and $\pi_{\alpha_i}(\beta)$ share a common vertex when $\alpha_{i+1}$ and $\beta$ are admissible for $\varepsilon$.
In the case that $\alpha_{i}$ and $\beta$ are disjoint (this is the last vertex of the sequence) we still deduce that all vertices in $\pi_{\alpha_i}(\beta)$ are adjacent to $\beta$.
Combining the above with Lemma~\ref{projection gives you uniformly bounded set}, this provides a path from $\alpha$ to $\beta$.

Finally, given any $\eta\in\pi_{\alpha'}(\beta')$, where $\alpha'\subset\alpha$ and $\beta'\subset\beta$ are admissible for $\varepsilon$, we recall that $\pi_{\alpha'}(\beta)\subset \pi_{\alpha'}(\beta')$ is a consequence of the definition. In light of Lemma~\ref{projection gives you uniformly bounded set}, this means $\eta$ has a path in the $8$-neighbourhood of $\mathcal{L}_\pitchfork(\alpha,\beta)$ to some vertex of $\pi_{\alpha'}(\beta)$. By the argument in the previous paragraph, there is a path to $\alpha$ and $\beta$ also. The lemma is proved.
\end{proof}

Finally, the connectivity of $\mathcal{C}^\dagger_\varepsilon(S^2)$ can be deduced. 
Indeed, if we are given arbitrary vertices $\alpha$ and $\beta$, then we can find $\alpha_0$ adjacent to $\alpha$, and by considering perturbations, we may assume in addition that $\alpha_0\pitchfork\beta$, and then invoke Lemma~\ref{coarsely connected} to find a path between $\alpha_0$ and $\beta$, and therefore a path between $\alpha$ and $\beta$.

\subsection{Proof of hyperbolicity}\label{Subsection-hyperbolicity}

\begin{definition}\label{minimal pair}
Let $\alpha,\beta\in\mathcal{C}^\dagger_\varepsilon(S^2)$, and let $\alpha'\subseteq \alpha$ and $\beta'\subseteq \beta$. Suppose that $\alpha'$ and $\beta'$ are admissible for $\varepsilon$. We say that ($\alpha'$,$\beta'$) is a \emph{minimal pair} if 
\begin{enumerate}
\item there is no choice of proper arc $\alpha''\subset \alpha'$ such that $\alpha''$ and $\beta'$ are admissible for $\varepsilon$, and
\item there is no choice of proper arc $\beta''\subset \beta'$ such that $\alpha'$ and $\beta''$ are admissible for $\varepsilon$.
\end{enumerate}
\end{definition}

\begin{lemma} \label{shrink to minimal pair}
    If $\alpha'$ and $\beta'$ are subarcs, and $\alpha'$ and $\beta'$ are admissible for $\varepsilon$, then there exists a minimal pair $(\alpha'',\beta'')$ such that $\alpha''\subseteq \alpha'$, $\beta''\subseteq \beta'$, and there is some equator $\eta\in\pi_{\alpha''}(\beta'')\cap \pi_{\alpha'}(\beta').$
\end{lemma}

\begin{proof}
    First we pick $\beta''$ a minimal compact connected subarc of $\beta'$ such that $\alpha'$ and $\beta''$ are admissible for $\varepsilon$. Then it is clear that $\pi_{\alpha'}(\beta')\subset \pi_{\alpha'}(\beta'')$ by definition.

    Now since $\beta''$ is minimal, it implies that there is a connected component of $b\subset\beta''\setminus\alpha'$ such that $(S^2\setminus (\alpha'\cup\beta''))\cup b$ has a connected component whose area is greater than $1-\varepsilon$ and hence greater than $\frac{1}{2}$.
    This implies that there is an equator $\eta\in \pi_{\alpha'}(\beta'')$ that intersects $\beta''$ twice transversely (only along $b$).
    In fact, extending $\beta''$ back to $\beta'$, we can pick $\eta$ so that it intersects each component of $\beta'\setminus\alpha'$ at most twice and transversely, and only intersects $\beta''$ twice along $b$, and so $\eta\in\pi_{\alpha'}(\beta')$.

    Now pick $\alpha''\subseteq\alpha'$ minimal such that $\alpha''$ and $\beta''$ are admissible for $\varepsilon$. 
    It is clear from the definition $\alpha''$ and $\beta''$ that $(\alpha'',\beta'')$ is a minimal pair. 
    Now observe that $\eta$ is disjoint from $\alpha''$ and intersects $\beta''$ twice transversely (only along $b$), so it is immediate that $\eta\in\pi_{\alpha''}(\beta'')$.
\end{proof}

\begin{lemma}\label{symmetric}
If $\alpha'$ and $\beta'$ are admissible for $\varepsilon$, and ($\alpha'$,$\beta'$) is a minimal pair then $\mathrm{diam}(\pi_{\alpha'}(\beta')\cup\pi_{\beta'}(\alpha'))\leq 18$.

\end{lemma}
\begin{proof}
First, $(\alpha',\beta')$ is never a minimal pair unless both $\alpha'$ and $\beta'$ are proper subarcs instead of $\varepsilon$-balanced curves.

So suppose both $\alpha'$ and $\beta'$ are proper closed subarcs. Since $(\alpha',\beta')$ is a minimal pair, we know that if we pick a connected subarc $\alpha''\subsetneqq \alpha'$ such that $|\alpha''\cap \beta'|=|\alpha'\cap \beta'|-1$, two connected components of $S^2\backslash (\alpha'\cup\beta')$, denoted by $A_1$ and $A_2$ respectively, will merge into one connected component of $S^2\backslash (\alpha''\cup\beta')$ with area greater than $1-\varepsilon$.
Similarly, if we pick a connected subarc $\beta''\subsetneqq \beta'$ such that $|\alpha'\cap \beta''|=|\alpha'\cap \beta'|-1$, two connected components of $S^2\backslash (\alpha'\cup\beta')$, denoted by $A_3$ and $A_4$ respectively, will merge into one connected component of $S^2\backslash (\alpha''\cup\beta')$ with area greater than $1-\varepsilon$. 
Since $\mathrm{area}(A_1)+\mathrm{area}(A_2)+\mathrm{area}(A_3)+\mathrm{area}(A_4)> 2-2\varepsilon\geq 1$, we know that $\{A_1,A_2\}\cap\{A_3,A_4\}\neq \emptyset$. So, there are two cases:

Case (i): $|\{A_1,A_2\}\cap\{A_3,A_4\}|=1$. Without loss of generality, we assume that $\{A_1,A_2\}\cap\{A_3,A_4\}=\{A_1\}$. The schematic diagram is shown in Figure \ref{fig: Case i}, where the blue arc is $\beta'$ and the red arc is $\alpha'$. And we write $\{\alpha^+,\alpha^-\}$ for the endpoints of $\alpha'$ and $\{\beta^+,\beta^-\}$ for the endpoints of $\beta'$. 
\begin{figure}[H]
    \centering
    \includegraphics[width=0.55\textwidth]{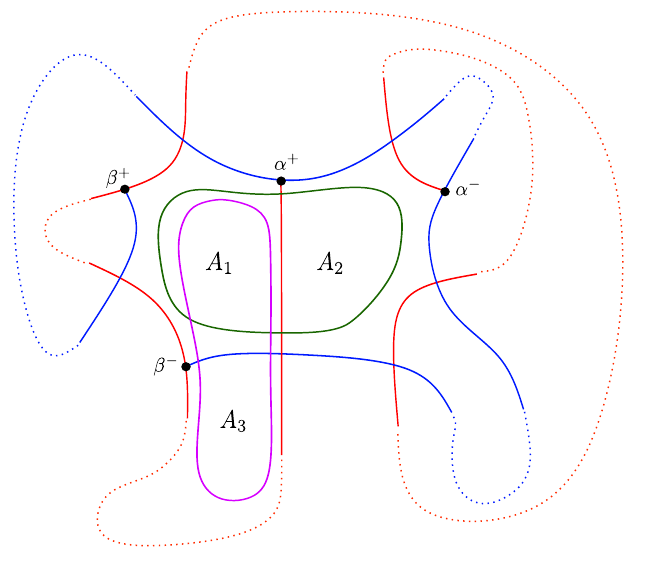}
    \caption{Case (i)}
    \label{fig: Case i}
\end{figure}
Without loss of generality, we assume that we shrink $\alpha'$ from the endpoint $\alpha^+$ such that $A_1$ and $A_2$ merge into one connected component and that we shrink $\beta'$ from the endpoint $\beta^-$ such that $A_1$ and $A_3$ merge into one connected component.

Since $\mathrm{area}(A_1)+\mathrm{area}(A_2)> 1-\varepsilon\geq \frac{1}{2}$ and $\mathrm{area}(A_1)+\mathrm{area}(A_3)> 1-\varepsilon\geq \frac{1}{2}$, we can find an equator $\gamma_1 \in \pi_{\beta'}(\alpha')$ (the green curve in Figure \ref{fig: Case i}) inside $\overline{A_1\cup A_2}$ and an equator $\gamma_2 \in \pi_{\alpha'}(\beta')$ (the purple curve in Figure \ref{fig: Case i}) inside $\overline{A_1\cup A_3}$ such that $\gamma_1\pitchfork \gamma_2$ and $|\gamma_1\cap\gamma_2|=2$.
It follows that $d_{\mathcal{C}^{\dagger}_{\varepsilon}(S^{2})}(\gamma_1,\gamma_2)=1$. Hence, by the triangle inequality, we have that 
$$
d(\pi_{\alpha'}(\beta'), \pi_{\beta'}(\alpha'))=\mathrm{diam}(\pi_{\alpha'}(\beta')\cup \pi_{\beta'}(\alpha'))\leq 8+1+8=17.
$$

Case (ii): $|\{A_1,A_2\}\cap\{A_3,A_4\}|=2$, i.e. $\{{A_1,A_2\}}=\{A_3,A_4\}$. 
The schematic diagram is shown in Figure \ref{fig: Case ii}, where the blue arcs are subarcs of $\beta'$ and the red arcs are subarcs of $\alpha'$. 
We write $q$ for one of the endpoints of $\alpha'$. And we assume that once we shrink $\alpha'$ from $q$ to get $\alpha''$, we will merge $A_1$ and $A_2$ to a connected component with $\mathrm{area}> 1-\varepsilon$. 
The endpoint of $\beta'$, which we shrink $\beta'$ to $\beta''$ from, written $p$, belongs to $\partial A_1$ and $\partial A_2$ respectively.
\begin{figure}[H]
    \centering
    \includegraphics[width=0.45\textwidth]{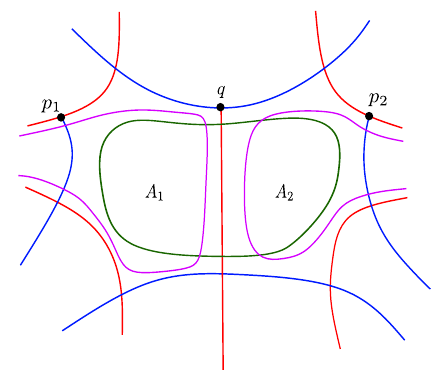}
    \caption{Case (ii)}
    \label{fig: Case ii}
\end{figure}
As shown in Figure \ref{fig: Case ii}, it is easy to find an equator $\gamma_2$ (the purple curve in Figure \ref{fig: Case ii}) inside the region $A_1\cup A_2$ such that $\gamma_1\pitchfork \gamma_2$ and $|\gamma_1\cap\gamma_2|\leq 4$. Moreover, $\gamma_2\in \pi_{\alpha'}(\beta')$ and $\gamma_1\in \pi_{\beta'}(\alpha')$. Hence, by Lemma~\ref{trivial intersections}, we know that $d_{\mathcal{C}^{\dagger}_{\varepsilon}(S^{2})}(\gamma_1,\gamma_2)\leq 2$. Therefore, by the triangle inequality, we have
$$
d(\pi_{\alpha'}(\beta'), \pi_{\beta'}(\alpha'))=\mathrm{diam}(\pi_{\alpha'}(\beta')\cup \pi_{\beta'}(\alpha'))\leq 2+8+8=18.
$$
\end{proof}

To prove the slimness condition in Lemma~\ref{Guessing Geodesic Lemma}, we also need the following lemma:
\begin{lemma}\label{lemma for the uniform boundedness of the guessing geodesic}
Let $\alpha,\beta,\gamma\in \mathcal{C}^{\dagger}_{\varepsilon}(S^2)$.
Let $\alpha'$ be a connected subarc of $\alpha$, and $\beta',\gamma'$ be connected closed subsets of $\beta,\gamma$ respectively. 
If $\alpha'$ and $\beta'$, and $\alpha'$ and $\gamma'$ are both admissible for $\varepsilon$, and $\beta'\pitchfork \gamma'$ with $|\beta'\cap \gamma'|\leq 2$, then $$\pi_{\alpha'}(\beta')\cap \pi_{\alpha'}(\gamma')\neq\emptyset.$$\end{lemma}

\begin{proof}
If $\gamma'=\gamma$ and $\gamma$ is disjoint from $\alpha'$, by definition of $\pi_{\alpha'}(\beta')$, we have $\gamma\in \pi_{\alpha'}(\beta')$. There is an (area-preserving) perturbation $\gamma''$ of $\gamma$ which intersects $\gamma$ exactly twice transversely and  $\gamma''\in \pi_{\alpha'}(\beta')$. We are done in this case.

\begin{figure}[H]
    \centering
    \includegraphics[width=0.7\textwidth]{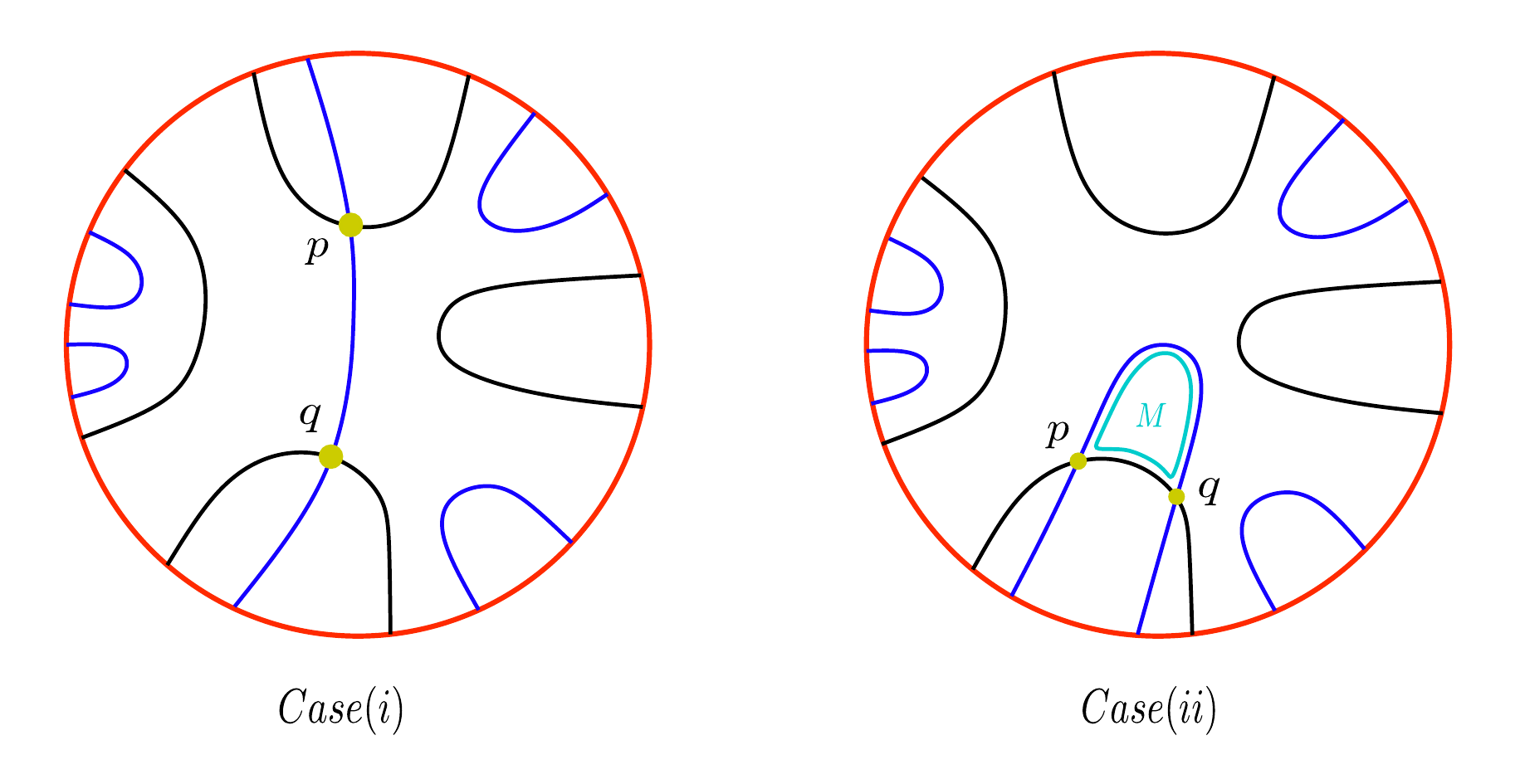}
    \caption{Schematic diagram}
    \label{fig: Preliminary lemma to prove slimness}
\end{figure}

Now we assume that $\beta'$ and $\gamma'$ are not disjoint from $\alpha'$ (otherwise we are done by the above argument after possibly swapping $\beta'$ and $\gamma'$).
To finish the proof, we claim that $$\pi_{\alpha'}(\beta')\cap\pi_{\alpha'}(\gamma')\neq\emptyset.$$

If there are connected components, one from each of $\beta'\setminus\alpha'$ and $\gamma'\setminus\alpha'$, that intersect twice, then there is a bigon $M$ between $\beta'$ and $\gamma'$ whose closure is disjoint from $\alpha'$.
See the right hand picture of Figure~\ref{fig: Preliminary lemma to prove slimness}. 
In Figure~\ref{fig: Preliminary lemma to prove slimness}, the blue arcs are connected components of $\beta'\backslash \alpha'$ and the black arcs are connected components of $\gamma'\backslash \alpha'$. 
It is not hard to see that there is a choice of equator intersecting $M$ that belongs to $\pi_{\alpha'}(\beta')\cap\pi_{\alpha'}(\gamma')$. The other cases are similar.
\end{proof}

\begin{lemma}\label{Slimness}
For any triple $\alpha,\beta,\gamma \in \mathcal{C}^{\dagger}_{\varepsilon}(S^2)$, the `geodesic guesses' $\mathcal{L}(\alpha,\beta)$, $\mathcal{L}(\beta,\gamma)$ and $\mathcal{L}(\gamma,\alpha)$ satisfy the conditions stated in Lemma~\ref{Guessing Geodesic Lemma}.\end{lemma}

\begin{proof}
First, we claim that if $d_{\mathcal{C}^{\dagger}_{\varepsilon}(S^2)}(\alpha,\beta)=1$, then $\mathrm{diam}(\mathcal{L}(\alpha,\beta))$ is uniformly bounded.
Let $\eta\in \pi_{\alpha_0'}(\beta_0')$ be an element in $\mathcal{L}_\pitchfork(\alpha_0,\beta_0)$.
We pick $\mu$ to be a $\varepsilon$-balanced curve that is a perturbation of $\beta$, such that it is adjacent to $\alpha,\beta,\beta_0$.
By considering a further perturbation of $\mu$, we may furthermore assume that $\mu\pitchfork\alpha_0$.
Since $\beta_0'\subseteq \beta_0$, it is the case that $\pi_{\alpha'_0}(\beta_0)\subseteq \pi_{\alpha'_0}(\beta_0')$.
Since $\beta_0\pitchfork \mu$ with $|\beta_0\cap \mu|\leq 2$, by Lemma~\ref{lemma for the uniform boundedness of the guessing geodesic}, we know that there is some equator $\eta_1\in \pi_{\alpha_0'}(\beta_0)\cap \pi_{\alpha_0'}(\mu)$.
By Lemma~\ref{projection gives you uniformly bounded set}, $\eta$ and $\eta_1$ are $8$-close.
Then pick $\mu'\subseteq \mu$ and $\alpha_0''\subseteq \alpha_0'$ such that $(\mu',\alpha_0'')$ is a minimal pair, using Lemma~\ref{shrink to minimal pair}, so there is some $\eta_2\in\pi_{\alpha_0''}(\mu')\cap\pi_{\alpha_0'}(\mu)$.
By Lemma~\ref{projection gives you uniformly bounded set}, $\eta_1$ and $\eta_2$ are $8$-close.
Then by Lemma~\ref{symmetric}, we know that $\mathrm{diam}(\pi_{\alpha_0''}(\mu')\cup \pi_{\mu'}(\alpha_0''))\leq 18$. 
By definition, we have that $\pi_{\mu'}(\alpha_0)\subset \pi_{\mu'}(\alpha_0'')$.
Now $\alpha_0$ is adjacent to $\alpha$, so by Lemma~\ref{lemma for the uniform boundedness of the guessing geodesic} we know there is $\eta_3\in\pi_{\mu'}(\alpha_0)\cap \pi_{\mu'}(\alpha)$. 
Therefore $\eta_2$ and $\eta_3$ are 18-close.
Finally, since $\alpha$ and $\mu$ are adjacent, any equator in $\pi_{\mu'}(\alpha)$ is within distance $2$ to $\alpha$, by Lemma~\ref{trivial intersections}, and so we are done by the triangle inequality.

We need to show the uniform slimness of the geodesic guesses, i.e. that $\mathcal{L}(\alpha,\beta)$ is contained in a uniform neighbourhood of $\mathcal{L}(\beta,\gamma)\cup\mathcal{L}(\gamma,\alpha)$. Note that this is immediately true if $\alpha,\beta,\gamma$ are not distinct. So suppose from now on that they are distinct.

It is enough to do the following: given any $\alpha_0'\subseteq \alpha_0$ and $\beta_0'\subseteq \beta_0$ with the following conditions
\begin{enumerate}
    \item $\alpha$ is adjacent to $\alpha_0$,
    \item $\beta$ is adjacent to $\beta_0$, 
    \item  $\alpha_0\pitchfork\beta_0$, and
    \item $\alpha_0'$ and $\beta_0'$ are admissible for $\varepsilon$,
\end{enumerate}
show that $\eta\in\pi_{\alpha_0'}(\beta_0')$ is uniformly close to either $\mathcal{L}(\alpha_0,\gamma_0)$ or $\mathcal{L}(\beta_0,\gamma_0)$, for some $\gamma_0$ adjacent to $\gamma$, such that $\gamma_0\pitchfork\alpha_0$ and $\gamma_0\pitchfork\beta_0$.

So to do this, suppose we are given such a curve $\eta$ as above. 
Then we pick $\gamma_0$ such that $\alpha_0\cap \beta_0 \cap \gamma_0=\emptyset$ and $\gamma_0$ intersects $\alpha_0,\beta_0$ transversely (if at all). 
Pick $\eta_1\in\pi_{\alpha_0''}(\beta_0'')\cap\pi_{\alpha_0'}(\beta_0')$ as in Lemma~\ref{shrink to minimal pair}.
Then $\eta$ and $\eta_1$ are $8$-close.
Pick arbitrary $\eta_2\in\pi_{\alpha_0''}(\beta_0'')\cup\pi_{\beta_0''}(\alpha_0'')$.
The diameter is at most $18$ by Lemma~\ref{symmetric}, so $\eta_1$ and $\eta_2$ are $18$-close.
Because there are no triple intersections of $\alpha_0,\beta_0,\gamma_0$, there exists a minimal compact subarc $\gamma_0'\subseteq\gamma_0$ such that (without loss of generality) $\gamma_0'$ and $\alpha_0''$ are admissible for $\varepsilon$ but $\gamma_0'$ and $\beta_0''$ are inadmissible for $\varepsilon$.
The without loss of generality here, by swapping the roles of $\alpha$ and $\beta$, is justified because we only care about the distance from the set $\pi_{\alpha_0''}(\beta_0'')\cup\pi_{\beta_0''}(\alpha_0'')$ to either $\mathcal{L}(\alpha_0,\gamma_0)$ or $\mathcal{L}(\beta_0,\gamma_0)$ at this point in the proof.

There exists an equator $\mu$ disjoint from $\beta_0''\cup\gamma_0'$ because $\beta_0''$ and $\gamma_0'$ are inadmissible for $\varepsilon$.
It is clear that $\mu$ and $\alpha_0''$ are admissible for $\varepsilon$ because $\mu$ is an equator.
By Lemma~\ref{lemma for the uniform boundedness of the guessing geodesic}, any vertex of $\pi_{\alpha_0''}(\beta_0'')$ is close to $\pi_{\alpha_0''}(\mu)$, and any vertex of $\pi_{\alpha_0''}(\mu)$ is close to $\pi_{\alpha_0''}(\gamma_0')$. This shows that $\eta$ is uniformly close to $\mathcal{L}(\alpha,\gamma)$, and we are done.
\end{proof}
Hence, we have verified that $\mathcal{C}^{\dagger}_{\varepsilon}(S^{2})$ is hyperbolic. In fact, they all have a hyperbolicity constant that is independent of $\varepsilon$.

\begin{remark}
Our methods still work if the vertices of the $\varepsilon$-balanced curve graph are allowed to be continuous (not necessarily smooth).
\end{remark}
\section{Constructing quasimorphisms}\label{section-BF quasimorphisms}
In this section, we're going to use Bestvina--Fujiwara \cite{bestvina2002bounded} to prove Theorem \ref{Q(Diff) is infinite dimensional}. 
We first need to find diffeomorphisms that act hyperbolically (i.e. with positive asymptotic translation length) and that are independent (i.e. there is a bound on how far their (oriented) quasi-axes fellow travel even after applying any area-preserving diffeomorphism to either axis). 
We will provide the rigorous definition of what it means for two hyperbolic elements to be independent.
We also make the following convention: In this section, $G$ denotes an abstract group and $X$ denotes a hyperbolic space such that $G$ acts on $X$ by isometries.

\subsection{Run the Bestvina--Fujiwara machine in the case of \texorpdfstring{$\Ham(S^2)$}{Ham(S2)}}
Recall that for $0<\varepsilon_{1}<\varepsilon_{2}\leq \frac{1}{2}$ the embedding $i\colon\mathcal{C}_{\varepsilon_{2}}^\dagger(S^2)\rightarrow \mathcal{C}_{\varepsilon_{1}}^\dagger(S^2)$ is 1-Lipschitz. 
According to Lemma~\ref{Witness for the big one but not for the small}, we pick a subsurface $W\subset S^2$, homeomorphic to $S_{0,n}$ for some $n\geq 4$, which is a witness for $\mathcal{C}^{\dagger}_{\varepsilon_2}(S^2)$ but with a boundary component in $\mathcal{C}^\dagger_{\varepsilon_1}(S^2)$.
According to Lemma~\ref{hyperbolics}, we can find $f\in \Ham(S^2)$ such that $f(W)=W$ and $f|_W$ represents a pseudo-Anosov pure mapping class of $W$.

\begin{lemma}\label{hyperbolic in two but elliptic in one}
$f$ acts hyperbolically (i.e. loxodromically) on $\mathcal{C}_{\varepsilon_{2}}^{\dagger}(S^{2})$. Moreover, $f$ acts elliptically on $\mathcal{C}_{\varepsilon_{1}}^{\dagger}(S^{2})$.

\end{lemma}
\begin{proof}

According to the proof of Theorem~\ref{infinite diameter}, we have that $d_{\mathcal{C}^{\dagger}_{\varepsilon}(S^{2})}(f^n(\alpha),\alpha)\geq \lfloor \frac{Cn}{4} \rfloor \geq n$. Hence, by definition, we have that $|f|>0$, i.e. $f$ is a hyperbolic (i.e. loxodromic) element of $\mathcal{C}_{\varepsilon_{2}}^{\dagger}(S^{2})$.
Moreover, $f$ fixes each boundary component of $W$, which as mentioned above, includes an  $\varepsilon_{1}$-balanced curve. Hence, $f$ fixes a point of $\mathcal{C}_{\varepsilon_{1}}^{\dagger}(S^{2})$, and so $f$ is an elliptic element as required.    
\end{proof}

By the same argument one can find an element $g$ acting hyperbolically but instead on $\mathcal{C}_{\varepsilon_{1}}^{\dagger}(S^{2})$. Since the embedding $i$ from $\mathcal{C}_{\varepsilon_{2}}^{\dagger}(S^{2})$ to $\mathcal{C}_{\varepsilon_{1}}^{\dagger}(S^{2})$ is $1$-Lipschitz, we can deduce that $g$ also acts hyperbolically on $\mathcal{C}_{\varepsilon_{2}}^{\dagger}(S^{2})$.

\begin{definition}\cite{bestvina2002bounded}\label{dependent}
Let $X$ be a $\delta$-hyperbolic geodesic metric space.
Let $g_{1},g_{2}\in G$ be two distinct hyperbolic elements with $(K,L)$-quasi-axes $A_{1}$ and $A_{2}$ in $X$. 
We say that $g_{1}$ and $g_{2}$ are dependent, denoted by $g_{1}\sim g_{2}$ if there exists $B=B(K,L,\delta)$, such that for an arbitrarily long segment $J$ in $A_{1}$, there exists an element $h\in G$ such that $h(J)$ is within the $B$-neighborhood of $A_{2}$ (preserving orientation). 
We say that $g_{1}$ and $g_{2}$ are \emph{independent} if $g_{1}\nsim g_{2}$. 
\end{definition}
\begin{remark}
$\sim$ is an equivalence relation. (see \cite{bestvina2002bounded} for details.)
\end{remark}

The following is \cite[Theorem~1]{bestvina2002bounded}.

\begin{theorem}\label{BestvinaFujiwara machinery}
Suppose a group $G$ acts on a hyperbolic graph $X$ by isometries. Suppose that $g_{1},g_{2}\in G$ act hyperbolically on $X$ and that $g_{1}\nsim g_{2}$. Then the space of homogeneous quasimorphisms on $G$ is infinite-dimensional.
\end{theorem}

\begin{lemma}\label{independent}
The maps $f,g\in \Ham(S^2)$ defined above satisfy that 
$f\nsim g$ when we look at their actions on $\mathcal{C}_{\varepsilon_{2}}^{\dagger}(S^{2})$.

\end{lemma}
\begin{proof}
We prove this by contradiction. Let $\alpha\in \mathcal{C}_{\varepsilon_{2}}^{\dagger}(S^{2})$. Assume $f\sim g$. We denote by $A_{f}=(f^{n}(\alpha))_{n\in \mathbb{Z}}$ the $K$-quasi-axis for $f$ and we denote by $A_{g}$ the $K'$-quasi-axis for $g$. Let $B=B(K,K',\delta)$ be the constant defining the relation $\sim$. Since $f$ acts elliptically on $\mathcal{C}_{\varepsilon_{1}}^{\dagger}(S^{2})$, we have that $i(A_{f})$ is a bounded set in $\mathcal{C}_{\varepsilon_{1}}^{\dagger}(S^{2})$, i.e. we have $\mathrm{diam}(i(A_{f}))=D$ for some constant $D$. On the other hand, $g$ acts hyperbolically on $\mathcal{C}_{\varepsilon_{1}}^{\dagger}(S^{2})$. Therefore, $i(A_{g})$ is still a bi-infinite quasi-geodesic in $\mathcal{C}_{\varepsilon_{1}}^{\dagger}(S^{2})$. For any segment $J$ in $A_{g}$, there exists an element $h\in \Ham(S^2)$ such that $hJ\subset N_{B}(A_{f})$. After taking the 1-Lipschitz embedding $i$, we know that $i(hJ)\subset N_{B}(i(A_{f}))$. Hence, $i(J)$ has diameter at most $2B+D$ in $\mathcal{C}_{\varepsilon_{1}}^{\dagger}(S^{2})$. Clearly, picking $J$ large enough, we obtain a contradiction. Therefore, $f$ and $g$ are independent.
\end{proof}
\noindent As a corollary of Theorem~\ref{BestvinaFujiwara machinery} and Lemma~\ref{independent}, we know that we can construct infinitely many homogeneous quasimorphisms on $\Ham(S^2)$. 
Moreover, according to Bestvina--Fujiwara's work \cite{bestvina2002bounded}, we can write down these quasimorphisms explicitly. Here, we just briefly recall what they did. For the details, we refer to \cite{bestvina2002bounded}. 

With the same assumption as Theorem~\ref{BestvinaFujiwara machinery}. Let $w$ be a finite (oriented) path in $X$. Let $|w|$ denote the length of $w$. For $g\in G$ the composition $g\circ w$ is a copy of $w$. Obviously $|g\circ w|=|w|$. 

Let $\alpha$ be a finite path. We define
$$
|\alpha|_{w}=\{\text{the maximal number of non-overlapping copies of $w$ in $\alpha$}\}.
$$
Suppose that $x,y\in X$ are two vertices and that $R$ is an integer with $0<R<|w|$. We define the integer 
$$
c_{w,R}(x,y)=d(x,y)-\inf_{\alpha}(|\alpha|-R|\alpha|_{w}),
$$
where $\alpha$ ranges over all paths from $x$ to $y$. Fix a basepoint $x_{0}\in X$, we define $h_{w}:G\rightarrow \mathbb{R}$ by
$$
h_{w}(g)= c_{w,R}(x_{0},g(x_{0}))-c_{w^{-1},R}(x_{0},g(x_{0})).
$$
In the proof of Theorem~\ref{BestvinaFujiwara machinery}, every homogeneous quasimorphism comes from the homogenisation of $h_{w}$, where $w$ is a word generated by $g_{1}$ and $g_{2}$. We know the homogenisation of $h_{w}$ is defined as $$\widetilde{h_{w}}(g)= \lim_{n\rightarrow\infty} \frac{h_{w}(g^n)}{n}.$$

\begin{lemma}\label{a useful property of BF quasimorphism}(\cite{fujiwara1998second})
Let $x_0\neq y_0\in X$. Let $h_w$ and $h'_w$ be the Bestvina-Fujiwara quasimorphisms with basepoints $x_0$ and $y_0$, respectively. Then there exists a constant $B>0$ such that $|h_w(g)-h'_w(g)|\leq B$ for every $g\in G$.
\end{lemma}
\begin{proof}
See~\cite[Lemma~3.8]{fujiwara1998second}.
\end{proof}

\begin{proposition}\label{homogenisation of BF quasimorphism is independent of the choice of the basepoint}
$\widetilde{h_{w}}$ is independent of the choice of the basepoint $x_0$.
\end{proposition}
\begin{proof}
This is an immediate consequence of Lemma~\ref{a useful property of BF quasimorphism}.
\end{proof}

Then we have the following:
\begin{proposition}\label{vanish on stabiliser of any balanced curve}
All homogeneous quasimorphisms $\widetilde{h_{w}}$ vanish on the stabiliser of any $\varepsilon$-balanced curve.
\end{proposition}
\begin{proof}
WLOG, we let $\alpha$ be an $\varepsilon$-balanced curve on $S^2$. And we denote the stabiliser of $\alpha$ by $\stab(\alpha)\coloneqq \{g\in \Ham(S^2) \mid g(\alpha)=\alpha\}$. Since $\widetilde{h_w}$ is independent of the choice of the basepoint by Proposition~\ref{homogenisation of BF quasimorphism is independent of the choice of the basepoint}, we can choose our basepoint to be $\alpha$. For any $g\in \stab(\alpha)$, it is easy to compute by definition that $\widetilde{h_w}(g)=0$.
\end{proof}

\subsection{\texorpdfstring{$C^0$-continuity of the quasimorphisms}{C0 continuity of the quasimorphisms}}

To prove the $C^0$-continuity of the quasimorphisms constructed in Theorem~\ref{Q(Diff) is infinite dimensional}, we need the following lemma proved by Entov--Polterovich--Py~\cite{entov2012continuity}:
\begin{lemma}(\cite{entov2012continuity})\label{Criterion for checking continuity}
Let $\mu\colon\Ham(S^2)\to \mathbb{R}$ be a homogeneous quasimorphism. Then $\mu$ is $C^0$-continuous if and only if there exists $a>0$ such that the following property
holds: For any disk $D\subset S^2$ of area less than $a$, the restriction of $\mu$ to the group $\Ham(D)$ vanishes.
\end{lemma}
\begin{proof}
See~\cite[Thoerem~3]{entov2012continuity}.
\end{proof}

\begin{corollary}\label{the quasimorphism we built is continuous}
All homogeneous quasimorphisms $\widetilde{h_{w}}$ are $C^{0}$-continuous on $\Ham(S^2)$.
\end{corollary}
\begin{proof}
Pick $a\in(0,\frac{1}{2})$. Then for any disk $D\subset S^2$ of area less than $a$, we can find an equator $L$ such that $\Ham(D)\subset \stab(L)\coloneqq \{f\in \Ham(S^2) \mid f(L)=L\}$. Then by Proposition~\ref{vanish on stabiliser of any balanced curve}, we know that $\widetilde{h_{\omega}}$ vanishes on $\Ham(D)$.
\end{proof}

To prove the rest of Theorem~\ref{Q(Diff) is infinite dimensional}, we require another lemma introduced by Entov--Polterovich--Py~\cite{entov2012continuity}:
\begin{lemma}\label{extension lemma}(\cite{entov2012continuity})
Let $\Lambda$ be a topological group and let $\Gamma\subset\Lambda$ be a dense subgroup. Any continuous homogeneous quasimorphism on $\Gamma$ extends to a continuous homogeneous quasimorphism on $\Lambda$.
\end{lemma}
\begin{proof}
See~\cite[Proposition~2]{entov2012continuity}.
\end{proof}

By Corollary~\ref{the quasimorphism we built is continuous} and Lemma~\ref{extension lemma}, we know that $\widetilde{h_{w}}$ can be extended continuously to $\mathrm{Homeo_{0}}(S^2,\omega)$ under the $C^0$-topology. With this, the proof of Theorem~\ref{Q(Diff) is infinite dimensional} is now fully established.

\begin{remark}
One could also construct $C^0$-continuous quasimorphisms on $\mathrm{Homeo_{0}}(S^2,\omega)$ by applying our method directly to a modified version of the $\varepsilon$-balanced curve graph, where we allow our $\varepsilon$-balanced curves to be continuous (not necessarily smooth) simple closed curves.
Since the proofs of the connectedness and hyperbolicity do not require any differentiability for the curves, the connectedness and hyperbolicity are still true for the modified graph.
Moreover, the group $\mathrm{Homeo_{0}}(S^2,\omega)$ acts on the modified graph by isometries.
The remaining part for running the Bestvina--Fujiwara machine is almost the same as the case of $\Ham(S^2)$.
\end{remark}
\section{An alternative `equator conjecture'}\label{section-Alternative equator theorem}
\begin{definition}\label{quantitative fragmentation metric}
Let $0<A<1$ and let $f\in\Ham(S^2)$. The $A$\emph{-quantitative fragmentation norm} $|f|_A$ is the minimal $N$ such that $f=h_1\dots h_N$, for some $h_1,\dots,h_N\in\Ham(S^2)$, such that each $h_i$ is supported on some disk of area at most $A$.

Let $\alpha$ and $\beta$ be equators of $S^2$. The $A$-quantitative fragmentation metric $d_A(\alpha,\beta)$ is the minimal possible value of the $A$-quantitative fragmentation norm $|f|_A$, over all Hamiltonian diffeomorphisms $f$ such that $f(\alpha)=\beta$.
\end{definition}

We now prove Theorem~\ref{alternative equator conjecture}, using Theorem~\ref{Q(Diff) is infinite dimensional}. The argument here, which relies on having a (coarsely) Lipschitz unbounded quasimorphism that vanishes on stabilisers of equators, follows along similar lines to that of Khanevsky's argument in \cite{khanevsky2009hofer} showing that the space of diameters has infinite Hofer metric.

In fact, our proof of Theorem~\ref{alternative equator conjecture} also applies to the orbit of an $\varepsilon$-balanced curve under $\Ham(S^2)$ or $\Homeo_0(S^2,\omega)$, which is implicit in the following proof.

\begin{proof}[Proof of Theorem~\ref{alternative equator conjecture}]
Given any $0<A<1$, we pick $\varepsilon\in (0,\frac{1}{2})$ sufficiently small such that $$\varepsilon<\min(A,1-A).$$
Fix any equator (or indeed, any $\varepsilon$-balanced curve) $\alpha$ of $S^2$.
Pick a non-trivial homogeneous quasimorphism $\varphi$ on $\Ham(S^2)$, as in Theorem~\ref{Q(Diff) is infinite dimensional}, and denote by $D=D(\varphi)$ the defect of $\varphi$.
Then given any $K>0$, there exists an element $f\in \Ham(S^2)$ such that $\varphi(f)\geq (K+1)D$.
By definition, we have
$$
d_A(\alpha,f\alpha)=\inf_{f\alpha=g\alpha} |g|_A.
$$
Now, whenever $g\alpha=f\alpha$, we have $g^{-1}f\alpha=\alpha$, which means that $g^{-1}f$ is in the stabiliser of $\alpha$.
Therefore, $\varphi(g^{-1}f)=0$, which implies that $|\varphi(g)-\varphi(f)|\leq D$.
Hence, $\varphi(g)\geq \varphi(f)-D$.

Now, let $N=|g|_A$, and so there exist disk-supported maps $h_1,\dots,h_N\in \Ham(S^2)$ with  area at most $A$ such that \[g=h_1\cdots h_N.\] 
We have $\varphi(h_i)=0$ for each $h_i$. This follows from the fact that each $h_i$ must stabilise some $\varepsilon$-balanced curve since we set $\varepsilon<\min(A,1-A)$.
Therefore, we obtain that 
$$
\varphi(g)\leq ND.
$$
Combining the above inequalities, we obtain
$$
|g|_A=N\geq \frac{\varphi(g)}{D}\geq \frac{\varphi(f)-D}{D}\geq K.
$$
Hence, we have $d_A(\alpha,f\alpha)\geq K$. But $K>0$ was arbitrary, so the diameter is infinite, as required.
\end{proof}

We should note here that if any of the quasimorphisms in Theorem~\ref{Q(Diff) is infinite dimensional} is Hofer-Lipschitz (for some choice of $\varepsilon\in (0,\frac{1}{2}]$), then a similar argument to the above will prove that the Hofer-diameter of the space of equators is infinite.

\printbibliography

\noindent{\sc Yongsheng JIA}\\
\noindent{\sc The University of Manchester, Department of Mathematics, M13 9PL, Manchester, UK}

\noindent{\it Email address:} {\tt \href{mailto:yongsheng.jia@manchester.ac.uk}{
yongsheng.jia@manchester.ac.uk}}

\vspace{1cm}

\noindent{\sc Richard Webb}\\
\noindent{\sc The University of Manchester, Department of Mathematics, M13 9PL, Manchester, UK}

\noindent{\it Email address:} {\tt \href{mailto:Richard.Webb@manchester.ac.uk}{Richard.Webb@manchester.ac.uk}}

\end{document}